\documentclass[12pt, twoside]{article}
\usepackage{amsmath,amsthm,amssymb}
\usepackage{times}
\usepackage{enumerate}
  \usepackage{paralist}
  \usepackage{graphics} %% add this and next lines if pictures should be in esp format
  \usepackage{epsfig} %For pictures: screened artwork should be set up with an 85 or 100 line screen
\usepackage{graphicx}
\usepackage{caption}
\usepackage{subcaption}
\usepackage{epstopdf}%This is to transfer .eps figure to .pdf figure; please compile your paper using PDFLeTex or PDFTeXify.
 \usepackage[colorlinks=true]{hyperref}
 \usepackage{multirow}

\def\titlerunning#1{\gdef\titrun{#1}}
\makeatletter
\def\author#1{\gdef\autrun{\def\and{\unskip, }#1}\gdef\@author{#1}}
\def\address#1{{\def\and{\\\hspace*{18pt}}\renewcommand{\thefootnote}{}%
\footnote {#1}}%
\markboth{\autrun}{\titrun}}
\makeatother
\def\email#1{e-mail: #1}

\def\keywords#1{\par\medskip
\noindent\textbf{Keywords.} pattern formation, predator--prey model, prey--taxis, stationary solutions, time periodic solutions, stability analysis}

\newtheorem{theorem}{Theorem}[section]
\newtheorem{corollary}{Corollary}
\newtheorem{lemma}[theorem]{Lemma}
\newtheorem{proposition}{Proposition}[section]

 \numberwithin{equation}{section}
\newtheorem{remark}{Remark}[section]

\numberwithin{equation}{section}

\begin{document}

%%%%% To ease editing, add:

\baselineskip=17pt

%%%%%%%%%%%%%%%%

%% In the running head, give an abbreviation of the title.
\titlerunning{Patterns of systems with prey--taxis}

\title{Stationary and time--periodic patterns of two--predator and one--prey systems with prey--taxis\thanks{Accepted by Discrete Contin. Dyn. Syst-Series A}}

\author{Ke Wang, Qi Wang and Feng Yu}

\date{Department of Mathematics,\\
 Southwestern University of Finance and Economics, \\
 555 Liutai Ave, Wenjiang, Chengdu, Sichuan 611130, China}

\maketitle

\address{K. Wang:
%Department of Mathematics, Southwestern University of Finance and Economics, 555 Liutai Ave, Wenjiang, Chengdu, Sichuan 611130, China;
\email{kkwang@2012.swufe.edu.cn}
\and
Q. Wang:
%Department of Mathematics, Southwestern University of Finance and Economics, 555 Liutai Ave, Wenjiang, Chengdu, Sichuan 611130, China;
\email{qwang@swufe.edu.cn. Corresponding author.  QW is supported by NSF-China (Grant 11501460) and the Project (No.15ZA0382) from Department of Education, Sichuan China}
\and
F. Yu:
%Department of Mathematics, Southwestern University of Finance and Economics, 555 Liutai Ave, Wenjiang, Chengdu, Sichuan 611130, China;
\email{yfeng@2012.swufe.edu.cn. Currently at Department of Mathematics, University of Central Florida, Orlando, USA}}

%\subjclass{Primary XXXX; Secondary YYYY}

%%%%%%%%

\begin{abstract}
This paper concerns pattern formation in a class of reaction--advection--diffusion systems modeling the population dynamics of two predators and one prey.  We consider the biological situation that both predators forage along the population density gradient of the preys which can defend themselves as a group.  We prove the global existence and uniform boundedness of positive classical solutions for the fully parabolic system over a bounded domain with space dimension $N=1,2$ and for the parabolic--parabolic--elliptic system over higher space dimensions.  Linearized stability analysis shows that prey--taxis stabilizes the positive constant equilibrium if there is no group defense while it destabilizes the equilibrium otherwise.  Then we obtain stationary and time--periodic nontrivial solutions of the system that bifurcate from the positive constant equilibrium.  Moreover, the stability of these solutions is also analyzed in detail which provides a wave mode selection mechanism of nontrivial patterns for this strongly coupled system.  Finally, we perform numerical simulations to illustrate and support our theoretical results.

%% Keywords are optional
\end{abstract}

%\bigskip
%\footnotesize
%\noindent\textit{Acknowledgments.}
%This research was partly supported by NSF (grant no. XXXX).

\section{Introduction}\label{section1}
One of the central problems in the study of ecological systems is to understand the spatial--temporal behaviors of population distributions of interacting species.  Over the past few decades, many mathematical models have been proposed and developed to investigate the collective influence of individual's dispersals on the spatial--temporal population distribution at a group level.  For example, reaction--diffusion equations have been applied to model the densities and spatial distributions of single or multiple interacting species in \cite{Fisher,MK,Murray,OL} etc.  In particular, the formation of nontrivial patterns in these systems such as diffusion--driven heterogeneity or Turing's pattern, especially those with concentration properties, can be used to describe the aggregation or segregation phenomena.

In this paper, we study the following $3\times3$ reaction--advection--diffusion system
\begin{equation}\label{11}
\left\{
\begin{array}{ll}
u_t=\nabla \cdot (d_1\nabla u -\chi u \phi(w) \nabla w)+\alpha_1(1-u)u+\beta_1uw,&x \in \Omega,t>0,\\
v_t=\nabla \cdot (d_2\nabla v -\xi v \phi(w) \nabla w)+\alpha_2(1-v)v+\beta_2vw,&x \in \Omega,t>0,\\
w_t=d_3\Delta w+\alpha_3(1-w)w-\beta_{31}uw-\beta_{32}vw,&   x \in \Omega,t>0,      \\
\partial_\textbf{n} u =\partial_\textbf{n} v =\partial_\textbf{n} w =0,&x\in \partial \Omega,t>0,\\
u(x,0)=u_0(x),v(x,0)=v_0(x),w(x,0)=w_0(x),&x\in\Omega,
\end{array}
\right.
\end{equation}
where $\Omega$ is a bounded domain in $\mathbb R^N$, $N\geq1$ with smooth boundary $\partial \Omega$ and unit outer normal $\textbf{n}$. $\nabla$ is the gradient operator and $\Delta$ is the Laplace operator.  $u$, $v$ and $w$ are functions of space--time variable $(x,t)$.  $d_i$, $\alpha_i$, $i=1,2,3$, and $\beta_i$, $\beta_{3i}$, $i=1,2$, are positive constants. $\chi$ and $\xi$ are also assumed to be positive constants and $\phi(w)$ is a smooth function.

(\ref{11}) models the population dynamics of three interacting species subject to Lotka--Volterra kinetics, where $u$ and $v$ are population densities of two predators at location $x$ and time $t$, and $w$ is the population density of the prey.  It is assumed that both predators consume the same prey species and disperse over the habitat by a combination of random diffusion and directed movement (prey--taxis) along the gradient of prey population density, while the preys move in the habitat randomly.  Diffusion rates $d_i$, $i=1,2,3$, measure the intensity of random dispersals of the species.  Here prey--taxis is the phenomena that predators with the ability to perceive the heterogeneity of prey distribution approach the patches with high preys density.  The positive constants $\chi$ and $\xi$ measure the intensity of the directed movement of each predator in response to the prey--taxis.  The prey--dependent function $\phi$ reflects the strength of prey--tactic movement of the predators with respect to the variation of prey population density.  Various specific forms of $\phi$ can be chosen, depending on the biological situation that one tries to model.  In particular, $\chi u \phi(w)>0$ represents the biological situation that predator $u$ forage the preys while $\chi u \phi(w)<0$ means that predator $u$ retreat the habitat of the preys which can defend themselves as a group when the population density is high.  See \cite{Hamilton,GMJW,XR} and the references therein for more examples and further discussions about antipredation of preys through group defense.  The population kinetics are of classical Lotka--Volterra type, where $\alpha_i$ are the intrinsic growth rates of the species which reproduce logistically, $\beta_i$ are the growth rates of the predators and $\beta_{3i}$ are the death rates of the preys due to predation.

In search of preys, prey--taxis is the process that predators move preferentially towards patches with high density of prey.  All predators forage preys in accordance with spatial distributions of prey density.  Though individual predator tends to forage the vicinity of recent captures, swarms of predators exhibit prey--taxis behaviors.  This uneven searching effort may result in both larger predation rate and aggregation of predators where preys are abundant.  For example, insects can find preys which are concentrated in some areas and sparse in others, though they lack long--distance sensory perception.  See \cite{KO} and the references therein for detailed discussions and field experiments.  This mechanism is the same as bacterial chemotaxis in which cellular organisms sense and move along the concentration gradient of stimulating chemicals in the environment \cite{HP,Ho,Ho3,KS}.

%\cite{CM,Curio,Hassell,HM1,HM2} etc. for more discussions and \cite{HL,Murdoch} for field experiments.
Various reaction--diffusion systems have been proposed to describe spatial preda\-tor--prey distributions under directed movements.  In \cite{KO}, Kareiva and Odell proposed a mechanistic approach, formulated as partial differential equations with spatially varying dispersals and advection, to demonstrate and explain that area--restricted search does create predator aggregation.  Since then, various reaction--diffusion systems have been proposed to model predator--prey dynamics with prey--taxis.  Sapoukhina \emph{et al.}  \cite{STA} assumed that such directed movement is determined by the velocity variation (i.e. the acceleration).  They investigated the effect of prey--taxis on the predator's ability to maintain pest population below a certain economic threshold value.  It is assumed in \cite{Cz,Gr,Tu} etc. that the directed movement of predator is due to the advective velocity.  We refer to \cite{HP,KS,Pt,WGY} etc. for the derivation or justification of (\ref{11}).

Reaction--diffusion models with prey--taxis subject to different population dynamics have been studied by various authors extensively.  Ainseba \emph{et al.} \cite{ABN} established the existence and uniqueness of weak solutions to prey--taxis models with volume--filling effect in prey--tactic sensitivity function.  Global existence and boundedness of classical solutions are obtained in \cite{HZ,Tao}, while nonconstant positive steady states are investigated in \cite{WSS,WYZ,WWZ} and \cite{LWS} using Crandall--Rabinowitz bifurcation theory and fixed point theory, respectively.  Lee \emph{et al.} \cite{LHL2} studied pattern formation in prey--taxis system under a variety of nonlinear functional responses. Travelling wave solutions for prey--taxis models are studied by the same trio in \cite{LHL1}.  Effects of prey--taxis on predator--prey models are investigated numerically in \cite{CSLR} which suggest that both response functions and initial data play important roles in pattern formation; moreover, predator--prey models admit chaotic patterns when the prey--taxis coefficient is sufficiently large.  We also want to mention that for predator--prey systems without prey--taxis, Okubo and Levin \cite{OL} pointed out that an Allee effect in the functional response and a density--dependent death rate of the predator are necessary for the pattern formation.  See the book of Murray \cite{Murray} for more works on prey--taxis models.  We also want to mention that there are also many works on the predator--prey models with density--dependent diffusion or the so--called cross--diffusions, such as \cite{Kuto, KutoY, NY, RA, ZKS} etc.

All the aforementioned works are devoted to studying prey--taxis models with one predator and one prey.  There are some works on two--predator and one--prey model without prey--taxis \cite{HDe,LKEF,PM,TH}.  Moreover, we refer the reader to \cite{JWZ,SKT,WGY} etc. for the studies on the dynamics of competitive reaction--diffusion systems with cross--diffusion or advection.  In \cite{LWZY}, Lin \emph{et al.} considered (\ref{11}) with $\chi=\xi=0$ over $\Omega=(0,L)$ with or without diffusion.  They investigated the global dynamics of all equilibria of the ODE system and obtained traveling wave solutions of the PDE system.  In this paper, we consider system (\ref{11}) with two predators and one prey, both predators foraging the preys prey--tactically.  We are motivated to study the effect of prey--taxis and group defense on the spatial--temporal dynamics of (\ref{11}), in particular, its positive solutions that exhibit stationary or time--oscillating spatial structures.

There are several scientific goals of our paper and the rest part of this paper is organized as follows.  In Section \ref{section2}, global existence of positive classical solutions are obtained for the full parabolic system in Theorem \ref{theorem26} over 1D and 2D bounded domains and for the parabolic--parabolic--elliptic system over arbitrary space dimensions in Theorem \ref{theorem27} respectively.  We also prove that these classical solutions are uniformly bounded in time.  In Section \ref{section3}, we study the linearized stability of the positive equilibrium to (\ref{11}).  It is shown that prey--taxis destabilizes the positive equilibrium if there is group defense in preys and it stabilizes the equilibrium otherwise.  Section \ref{section4} and Section \ref{section5} are devoted to the steady state and Hopf bifurcation analysis of (\ref{11}) over $\Omega=(0,L)$.  Existence and stability of stationary and time--periodic spatial patterns are established in Theorem \ref{theorem42}, Theorem \ref{theorem43} and Theorem \ref{theorem52}.  Though there have been many works devoted to the study of nonconstant steady states of the 2$\times$2 prey--taxis systems, regular time--periodic spatial patterns are quite new to the literature.  Finally, we perform extensive numerical studies in Section \ref{section6} to illustrate and support our theoretical findings.  We would like to note that, $C$ and $C_i$ are assumed to be positive constants that may vary from line to line in the sequel.

\section{Existence and boundedness of global solutions}\label{section2}
In this section, we study the global existence and boundedness of classical positive solutions to (\ref{11}).  There are two well--established methods in proving the global boundedness for reaction--diffusion systems in the literature.  One is to construct its time--monotone Lyapunov--functional and the other is to use Gagliardo--Nirenberg type estimate through Moser--$L^p$ iteration.  As we shall see in our mathematical analysis and numerical simulations, (\ref{11}) admits time--periodic spatially inhomogeneous solutions for properly chosen parameters and hence lacks time--monotone Lyapunov--functional and maximum principle.  Our proof of global existence and boundedness of classical positive solutions to (\ref{11}) is based on the local theory of Amann \cite{Am1,Am2} and the Moser--Alikakos $L^p$ iteration technique \cite{A0}.

\subsection{Local existence and preliminary results}
We first obtain the local existence and uniqueness of positive classical solutions to (\ref{11}) and their extensibility criterion based on Amann's theory \cite{Am2}.  To this end, we convert it into the following triangular form
\begin{equation}\label{21}
\left(
\begin{array}{c}u\\
v\\
w
\end{array}
\right)_t=
\nabla \cdot \left[ \mathcal{D}_0 \nabla \left(
\begin{array}{c}u\\
v\\
w
\end{array}
\right) \right]
+\left(
\begin{array}{c}
\alpha_1(1-u)u+\beta_1uw\\
\alpha_2(1-v)v+\beta_2vw\\
\alpha_3(1-w)w-\beta_{31}uw-\beta_{32}vw
\end{array}
\right),
\end{equation}
with
\begin{equation*}\mathcal{D}_0 =\begin{pmatrix}
 d_1  &  0& -\chi u \phi(w)  \\
  0    & d_2& -\xi v \phi(w)\\
  0       & 0    &  d_3
  \end{pmatrix},
  \end{equation*}
subject to nonnegative initial data $u_0,v_0,w_0$ and homogeneous Neumann boundary conditions.  Since all the eigenvalues of $\mathcal{D}_0$ are positive, system (\ref{21}) is normally parabolic, and we have from the standard parabolic maximum principles that $u,v,w\geq0$ in $\Omega\times\mathbb R^+$.  Moreover, the following results are evident from Theorem 7.3 and Theorem 9.3 of \cite{A}, Theorem 5.2 in \cite{Am2} and the standard parabolic regularity arguments.
\begin{theorem}\label{theorem21}
Let $\Omega$ be a bounded domain in $\mathbb R^N$, $N\geq1$ with smooth boundary $\partial \Omega$.  Assume that $d_i$, $\alpha_i$ for $i=1,2,3$, and $\beta_i$, $\beta_{3i}$ for $i=1,2$ are positive constants.  Suppose that the initial data $(u_0, v_0, w_0)\in  C^1(\bar \Omega)\times  C^1(\bar \Omega)\times W^{2,p}$ for some $p>N$, and $u_0,v_0,w_0\geq0$, $\not \equiv 0$ in $\Omega$.  Then there exist a constant $T_{\max}\in(0,\infty]$ and a unique solution $(u(x,t),v(x,t),w(x,t))$ to (\ref{11}) which is nonnegative on $\bar \Omega\times [0,T_{\max})$ such that $(u(\cdot,t),v(\cdot,t),w(\cdot,t)) \in C([0,T_{\max}),H^1(\Omega)\times H^1(\Omega)\times H^1(\Omega))$ and $(u,v,w)\in C^{2+\alpha,1+\frac{\alpha}{2}}_{loc}(\bar \Omega \times (0,T_{\max}))^3$ for any $0<\alpha<\frac{1}{4}$; moreover, if $\sup_{s\in(0,t)}\Vert (u,v,w)(\cdot,s) \Vert_{L^\infty}$ is bounded for $t\in(0,T_{\max}]$, then $T_{\max}=\infty$, i.e., $(u,v,w)$ is a global solution to (\ref{11}).  Furthermore, $(u,v,w)$ is a classical solution and $(u,v,w) \in C^\alpha((0,\infty),C^{2(1-\beta)}(\bar \Omega)\times C^{2(1-\beta)}(\bar \Omega)\times C^{2(1-\beta)}(\bar \Omega))$ for any $0\leq \alpha \leq \beta \leq 1$.
\end{theorem}
According to Theorem \ref{theorem21}, the local solutions $(u,v,w)$ are global if their $L^\infty$--norms are bounded in time.  To establish the global existence, we state some basic properties of the local solutions.
\begin{lemma}\label{lemma22}
Suppose that all the conditions in Theorem \ref{theorem21} hold and $(u_0, v_0, w_0)\in  C^1(\bar \Omega)\times  C^1(\bar \Omega)\times W^{2,p}$ for some $p>N$.  Let $(u,v,w)$ be the classical solution of (\ref{11}) over $\Omega\times (0,T_{\max})$.  Then there exist positive constants $C_1$ and $C_2$ such that
\begin{equation}\label{22}
\Vert u(\cdot,t) \Vert_{L^1(\Omega)}+\Vert v(\cdot,t) \Vert_{L^1(\Omega)} \leq C_1=C(\Vert u_0 \Vert_{L^1},\Vert v_0 \Vert_{L^1},a_1,b_1, \vert \Omega \vert), \forall t\in(0,T_{\max}),
\end{equation}
and
\begin{equation}\label{23}
0\leq w(x,t)\leq C_2= C(\Vert w_0 \Vert_{\infty},\alpha_3),\forall t\in(0,T_{\max});
\end{equation}
moreover, $w(x,t)\in(0,1)$ on $\bar \Omega\times (0,T_{\max})$ if $w_0(x) \in [0,1]$ on $\bar \Omega$.
\end{lemma}

\begin{proof}
We first prove (\ref{23}).  Let $\bar{w}(t)$ be the solution of
\begin{equation}\label{24}
\frac{d \bar{w}(t)}{dt}=\alpha_3(1- \bar{w}(t))\bar{w}(t),\bar{w}(0)=\max_{\bar\Omega} w_0(x),
\end{equation}
then it is easy to see that $\bar{w}(t)$ is uniformly bounded and it is a super-solution to the third equation of (\ref{11}).  Hence we have that $w(x,t)\leq \bar{w}(t)$ from the maximum principle and this gives rise to (\ref{23}).  Moreover, if $w_0(x)\in(0,1)$, $\hat w(t)\equiv1$ is a super--solution and this implies that $w(x,t)\in(0,1)$.

We next prove the boundedness of $\Vert u(\cdot,t)\Vert_{L^1}$ and the same argument applies for $\Vert v(\cdot,t)\Vert_{L^1}$.  Integrating the first equation in (\ref{11}) over $\Omega$, we have from the boundedness of $\Vert w(\cdot,t)\Vert_{L^\infty}$ that
\[\frac{d}{dt} \int_\Omega u(x,t) \leq \Big(\alpha_1+\beta_1\Vert w(\cdot,t)\Vert_{L^\infty}\Big) \int_\Omega u(x,t)-\alpha_1 \int_\Omega u^2(x,t),\]
here and in the rest of this section we skipped $dx$.  In light of $(\int_\Omega u(x,t) )^2\leq |\Omega|\int_{\Omega}u^2(x,t)$, we have
\begin{align*}
  \frac{d}{dt}\int_\Omega u(x,t) \leq (\alpha_1+\beta_1 \Vert w(\cdot,t)\Vert_{L^\infty})\int_\Omega u(x,t) -\frac{\alpha_1}{|\Omega|}\Big(\int_\Omega u(x,t) \Big)^2.
\end{align*}
Solving this ordinary differential inequality gives
\begin{align*}
\int_\Omega u(x,t) \leq \max{ \Big\{ \int_{\Omega}u_0(x), |\Omega|\Big( 1+\frac{\beta_1}{\alpha_1} \sup_{0\leq t<T_{\max{}}}\Vert w(\cdot,t)\Vert_{L^\infty} \Big)      \Big\} }.
\end{align*}
\end{proof}
In order to estimate $L^p$--norms of $u$ and $v$, we need to provide an \emph{a priori} estimate on $\Vert \nabla w\Vert_{L^\infty}$.  The following lemma is due to the well--known smoothing properties of operator $-d_3\Delta+1$ and embeddings between the analytic semigroups generated by $\{e^{t\Delta}\}_{t>0}$.  We refer the reader to \cite{HW,Winkler} for references.
\begin{lemma}\label{lemma23}
Suppose that $\Omega$ is a bounded domain in $\mathbb R^N$, $N\geq1$.  Assume that all the conditions in Theorem \ref{theorem21} hold and let $(u,v,w)$ be the classical solution of (\ref{11}) over $(0,T_{\max})$.  Then there exists a constant $C$ dependent on $\Vert w_0 \Vert _{W^{1,q}(\Omega)}$ and $\vert \Omega\vert$ such that for $p\in(1,\infty)$
\begin{equation}\label{25}
\Vert w(\cdot,t) \Vert_{W^{1,q}} \le C\Big(1+\sup_{s\in(0,t)} \big(\Vert u(\cdot,s)\Vert_{L^p}+ \Vert v(\cdot,s)\Vert_{L^p}\big)\Big), \forall t\in (0,T_{\max}),
\end{equation}
for any $q\in(1,\frac{Np}{N-p})$ if $p\in [1,N)$, $q\in (1,\infty)$ if $p=N$ and $q=\infty$ if $p>N$.
\end{lemma}
\begin{proof}
We first write $w$--equation into the following variation--of--constants formula
\begin{equation}\label{26}
w(\cdot,t)=e^{ (d_3\Delta-1)t}w_0+\int_0^t e^{(d_3\Delta-1)(t-s)} h(u(\cdot,s),v(\cdot,s),w(\cdot,s)) ds,
\end{equation}
where $h(u,v,w)=w+\alpha_3(1-w)w-\beta_{31}uw-\beta_{32}vw$.  Applying Lemma 1.3 in \cite{Winkler} on (\ref{26}) and using (\ref{23}), we see that there exists a positive constant $C$ such that for $1\leq p, q \leq \infty$,
\begin{equation}\label{27}
 \Vert w(\cdot,t) \Vert _{W^{1,q}} \!\leq \!C\!\left(1\! + \!\int_0^t\!e^{-\nu(t-s)} (t-s)^{-\frac{1}{2}-\frac{N}{2}(\frac{1}{p}-\frac{1}{q})} \big(\Vert u(\cdot,s)\Vert_{L^p}+\Vert v(\cdot,s) \Vert_{L^p}\big)ds\right),
\end{equation}
where $\nu$ is the first Neumann eigenvalue of $-d_3\Delta+1$.  On the other hand, we know from the gamma function that
\[ \sup_{t\in(0,\infty)}\int_0^t e^{-\nu(t-s)} (t-s)^{-\frac{1}{2}-\frac{N}{2}(\frac{1}{p}-\frac{1}{q})}  ds<\infty, \text{ since } -\frac{1}{2}-\frac{N}{2}(\frac{1}{p}-\frac{1}{q})>-1,\]
therefore (\ref{25}) follows from (\ref{27}).
\end{proof}
In order to apply the Moser--iteration, it is sufficient to prove the boundedness of $\Vert \nabla w(\cdot,t) \Vert _{L^{2(N+1)}}$ thanks to the quadratic decay kinetics in (\ref{11}).
\begin{proposition}\label{proposition21}
Let $\Omega$ be a bounded domain in $\mathbb R^N$, $N\ge1$, and $(u,v,w)$ be the classical solutions of (\ref{11}) over $(0,T_{\max})$.  If there exists a constant $C_1$ such that
\[\Vert \nabla w(\cdot,t) \Vert _{L^{2(N+1)}}\le C_1, \forall t \in(0,T_{\max}),\]
then there exists a constant $C_2$ such that
\[\Vert u(\cdot,t) \Vert _{L^\infty}+\Vert v(\cdot,t) \Vert _{L^\infty}\le C_2, \forall t \in(0,T_{\max}).\]
Therefore $(u,v,w)$ is a global solution to (\ref{11}) over $\Omega\times (0,\infty)$.
\end{proposition}
\begin{proof}
For each $p\geq1$, we have from straightforward calculations
\begin{align}\label{28}
\frac{1}{p}\frac{d}{dt}\int_{\Omega} u^p =&\int_\Omega u^{p-1}\nabla\cdot(d_1\nabla u-\chi u\phi(w)\nabla w)+\int_\Omega\alpha_1(1-u)u^p+\beta_1u^pw \nonumber\\
=&-\frac{4d_1(p-1)}{p^2} \int_{\Omega} \vert\nabla u^{\frac{p}{2}} \vert^2+\frac{2\chi(p-1)}{p} \int_\Omega u^{\frac{p}{2}}\phi(w)\nabla u^{\frac{p}{2}} \cdot\nabla w \nonumber\\
&+\int_\Omega\alpha_1(1-u)u^p+\beta_1u^pw.
\end{align}
To estimate the second integral in (\ref{28}), we have from Young's inequality that, for any constant $\epsilon>0$, there exists $C(\epsilon)>0$ such that
\[u^{\frac{p}{2}}\nabla u^{\frac{p}{2}} \cdot  \nabla w \leq \epsilon \vert\nabla u^{\frac{p}{2}} \vert^2+\epsilon (u^\frac{p}{2})^\frac{2(p+1)}{p}+C(\epsilon) \vert\nabla w\vert^{2(p+1)}.\]
On the other hand, for each $p>0$ there exists $C_0>0$ such that $\alpha_1(1-u)u^p+\beta_1u^pw<-\frac{\alpha_1}{2}u^{p+1}+C_0$.  For the simplicity of notations and without loss of our generality, we assume that $\Vert \phi(w)\Vert_{L^\infty}\leq 1$ in light of (\ref{23}).  Choosing $p=N$ in (\ref{28}), thanks to the boundedness of $\Vert \nabla w(\cdot,t)\Vert_{L^{2(N+1)}}$ we have
\begin{align}\label{29}
&\frac{1}{N}\frac{d}{dt}\int_{\Omega} u^N  \nonumber \\
\le&-\frac{4d_1(N-1)}{N^2} \int_{\Omega} \vert\nabla u^{\frac{N}{2}} \vert^2  +\frac{2\chi(N-1)}{N} \int_\Omega u^{\frac{N}{2}}\nabla u^{\frac{N}{2}} \cdot  \nabla w  -\frac{\alpha_1}{2}\int_\Omega u^{N+1} + C_0\vert \Omega \vert \nonumber \\
\le&  -\frac{4d_1(N-1)}{N^2} \int_{\Omega} \vert\nabla u^{\frac{N}{2}} \vert^2 +\frac{2\chi(N-1)}{N} \int_\Omega\Big( \epsilon u^{N+1}+\epsilon \vert\nabla u^{\frac{N}{2}} \vert^2 \Big)\nonumber \\
&+\frac{2\chi(N-1)C(\epsilon)}{N}\int\vert\nabla w\vert^{2(N+1)} -\frac{\alpha_1}{2}\int_\Omega u^{N+1}+C(\epsilon)  \nonumber \\
\le&  \Big(-\frac{4d_1(N-1)}{N^2}+\frac{2\chi(N-1)\epsilon}{N}\Big) \int_{\Omega} \vert\nabla u^{\frac{N}{2}} \vert^2\nonumber \\&+\Big(\frac{2\chi(N-1)\epsilon}{N}-\frac{\alpha_1}{2}\Big) \int_\Omega  u^{N+1} +C(\epsilon) \nonumber \\
\le&  -\frac{2d_1(N-1)}{N^2} \int_{\Omega} \vert\nabla u^{\frac{N}{2}} \vert^2-\frac{\alpha_1}{4} \int_\Omega u^{N+1}+C(\epsilon),
\end{align}
where $\epsilon$ is chosen to be sufficiently small.  Then we can have from (\ref{29}) that $\Vert u(\cdot,t)\Vert_{L^N}\leq C$, $\forall t\in(0,T_{\max})$ and this, in light of Lemma \ref{lemma23}, implies that $\Vert \nabla w(\cdot,t)\Vert_{L^q}\leq C$ for any $q\in[1,\infty)$.

In particular, we note that $\Vert \nabla w(\cdot,t)\Vert_{L^{2(N+2)}}$ is bounded, therefore we can choose $p=N+1$ in (\ref{28}) and perform the same calculations as in (\ref{29}) to show that $\Vert u(\cdot,t)\Vert_{L^{N+1}}$ is uniformly bounded.  Once again by applying Lemma \ref{lemma23}, we can show that $\Vert \nabla w(\cdot,t)\Vert_{L^\infty}$ is bounded.

Without losing the generality of our analysis, we assume that $\Vert\phi(w) \nabla w\Vert_{L^\infty}\leq 1$, and then we have from (\ref{28})
{\small\[\frac{1}{p}\frac{d}{dt}\int_{\Omega} u^p\leq -\frac{4d_1(p-1)}{p^2} \int_{\Omega} \vert\nabla u^{\frac{p}{2}} \vert^2+\frac{2\chi(p-1)}{p} \int_\Omega u^{\frac{p}{2}}\vert \nabla u^{\frac{p}{2}}\vert+a_1^*\int_\Omega u^p-\alpha_1\int_\Omega u^{p+1},\]}where $\alpha_1^*=\alpha_1+\beta_1\sup_{t\geq0}\Vert w(\cdot,t)\Vert_{L^\infty}$.  We can further estimate it
\begin{align}\label{210}
\frac{1}{p}\frac{d}{dt}\int_{\Omega} u^p
&\le  -\frac{4d_1(p-1)}{p^2} \int_{\Omega} \vert\nabla u^{\frac{p}{2}} \vert^2 +\frac{2\chi(p-1)}{p} \int_\Omega\Big(\epsilon \vert\nabla u^{\frac{p}{2}} \vert^2+\frac{1}{4\epsilon} u^p\Big)+\alpha_1^*\int_\Omega u^p \nonumber \\
&\le  \Big(-\frac{4d_1(p-1)}{p^2}+\frac{2\chi(p-1)\epsilon}{p}\Big) \int_{\Omega} \vert\nabla u^{\frac{p}{2}}+\Big(\frac{\chi(p-1)}{2p\epsilon}+\alpha_1^*\Big) \int_\Omega  u^p \nonumber \\
&\le   -\frac{2d_1(p-1)}{p^2} \int_{\Omega} \vert\nabla u^{\frac{p}{2}} \vert^2+\Big(\frac{\chi^2(p-1)}{2d_1}+\alpha_1^*\Big)\int_\Omega  u^p,
\end{align}
by taking $\epsilon$ smaller than $\frac{d_1}{p\chi}$. Letting $p$ go to $\infty$ and applying the standard Moser--Alikakos $L^p$--iteration \cite{A0} on (\ref{210}), we can show that there exists a constant $C>0$ such that $\Vert u(\cdot,t) \Vert_{L^\infty}\leq C, \forall t\in(0,T_{\max})$.  Similarly, we can prove the uniform boundedness of $\Vert v(\cdot,t) \Vert_{L^\infty}$.  The proof of this proposition completes.
\end{proof}

\subsection{\emph{A priori} estimates of fully parabolic system for $N\le2$}
According to Lemma \ref{lemma23} and Proposition \ref{proposition21}, in order to obtain the global existence of (\ref{11}), it is sufficient to establish the boundedness of $u$ and $v$ in their $L^p$--norms for some $p\geq N$.  Lemma \ref{lemma22} already gives $L^1$ boundedness from which global existence of (\ref{11}) in 1D follows.  In this section, we restrict our attention to study (\ref{11}) over 2D and we want to prove the boundedness of $\Vert (u,v)(\cdot,t)\Vert_{L^2}$.  First of all, we introduce several entropy--type inequalities to estimate weighted functions $\Vert u\ln u\Vert_{L^1}$, $\Vert v\ln v\Vert_{L^1}$ and $\Vert \nabla w \Vert_{L^2}$, etc.  The \emph{a priori} estimates rely on Gagliardo--Nirenberg interpolation inequality for which \cite{NSY,TaoW} are good references.  We now give the following lemma.
\begin{lemma}\label{lemma24}
Let $\Omega \subset \mathbb R^2$ be a bounded domain with smooth boundary $\partial \Omega$.  Suppose that all the conditions in Theorem \ref{theorem21} are satisfied and $(u,v,w)$ are the local solutions to (\ref{11}) on $\Omega \times (0,T_{\max})$.  Then there exists a positive constant $C$ independent of $T_{\max}$ such that
\begin{equation}\label{211}
\Vert u \ln u \Vert_{L^1(\Omega)}+\Vert v \ln v \Vert_{L^1(\Omega)}+\Vert\nabla w \Vert_{L^2(\Omega)} <C, \forall t\in(0,T_{max}).
\end{equation}
\end{lemma}
\begin{proof}
Using $u$--equation of (\ref{11}) and the boundedness of $\Vert\phi(w)\Vert_{L^\infty}$, we have from integration by parts and Young's inequality
\begin{align}\label{212}
&\frac{d}{dt} \int_\Omega u\ln u=\int_\Omega(\ln u+1)u_t \nonumber \\
=&\int_\Omega \nabla \cdot(d_1 \nabla u-\chi u \phi(w)\nabla w)(\ln u+1)+\int_\Omega(\alpha_1-\alpha_1u+\beta_1w)u (\ln u+1)  \nonumber \\
=&-\int_\Omega (d_1 \nabla u-\chi u \phi(w)\nabla w)\cdot \frac{\nabla u}{u}+\int_\Omega(\alpha_1-\alpha_1u+\beta_1w)u (\ln u+1)  \nonumber \\
\leq &-d_1\int_\Omega\frac{\vert\nabla u \vert^2}{ u}+\chi \int_\Omega \phi(w) \sqrt u\frac{\nabla u}{\sqrt u} \cdot\nabla w -\frac{\alpha_1}{2}\int_\Omega u^2\ln u+C_1, \nonumber \\
\leq& -\frac{d_1}{2}\int_\Omega\frac{\vert\nabla u \vert^2}{u}-\frac{\alpha_1}{4}\int_\Omega u^2\ln u+C_2\int_\Omega \vert\nabla w \vert^4+C_3,
\end{align}
where $C_1$, $C_2$ and $C_3$ are positive constants.  Similarly we obtain from the $v$--equation
\begin{equation}\label{213}
\frac{d}{dt} \int_\Omega v\ln v\leq-\frac{d_2}{2}\int_\Omega\frac{\vert\nabla v \vert^2}{v}-\frac{\alpha_2}{4}\int_\Omega v^2\ln v+C_4\int_\Omega \vert\nabla w \vert^4+C_5.
\end{equation}

On the other hand, we have from straightforward calculations
\begin{align}\label{214}
&\frac{1}{2}\frac{d}{dt} \int_\Omega \vert \nabla w\vert^2=\int_\Omega\nabla w \cdot \nabla w_t=- \int_\Omega\Delta w w_t\nonumber \\
=&- \int_\Omega\Delta w \big(d_3\Delta w+\alpha_3(1-w)w-\beta_{31}uw-\beta_{32}vw\big)\nonumber \\
&\leq-d_3\int_\Omega \vert\Delta w\vert^2+\int_\Omega \vert\Delta w\vert(C_6u+C_7v+C_8)\nonumber\\
&\leq -\frac{d_3}{2} \int_\Omega \vert\Delta w\vert^2+C_9\int_\Omega u^2+C_{10}\int_\Omega v^2+C_{11},
\end{align}
where $C_i$'s are positive constants.  In light of (\ref{23}), we have from Gagliardo--Nirenberg interpolation that there exists $C_*>0$ such that
\begin{equation}\label{215}
\Vert \nabla w \Vert^4_{L^4}\leq C_*(\Vert \Delta w \Vert^2_{L^2}+1).
\end{equation}

Multiplying (\ref{215}) by $\mu_1$ and then adding it to (\ref{212}) and (\ref{213}), we have
\begin{align*}
&\frac{d}{dt} \Big(\int_\Omega u\ln u+v\ln v+\frac{\mu_1}{2}|\nabla w|^2\Big)\\
\leq& -\frac{d_1}{2}\int_\Omega\frac{\vert\nabla u \vert^2}{ u}-\frac{\alpha_1}{4}\int_\Omega u^2\ln u+C_2\int_\Omega \vert\nabla w \vert^4+C_3%\\
\end{align*}\begin{align*}&-\frac{d_2}{2}\int_\Omega\frac{\vert\nabla v \vert^2}{v}-\frac{\alpha_2}{4}\int_\Omega v^2\ln v+C_4\int_\Omega \vert\nabla w \vert^4+C_5\\
&-\frac{d_3\mu_1}{2} \int_\Omega \vert\Delta w\vert^2+C_9\int_\Omega u^2+C_{10}\int_\Omega v^2+C_{11}\\
\leq &-\frac{d_1}{2}\int_\Omega\frac{\vert\nabla u \vert^2}{ u} -\frac{d_2}{2}\int_\Omega\frac{\vert\nabla v \vert^2}{v}-\frac{\alpha_1}{4}\int_\Omega u^2\ln u-\frac{\alpha_2}{4}\int_\Omega v^2\ln v\\
&+(C_2+C_4)\int_\Omega \vert\nabla w \vert^4-\frac{d_3\mu_1}{2} \int_\Omega \vert\Delta w\vert^2+C_{12}  \\
\leq& -\frac{\alpha_1}{4}\int_\Omega u^2\ln u-\frac{\alpha_2}{4}\int_\Omega v^2\ln v+(C_2+C_4)\int_\Omega \vert\nabla w \vert^4-\frac{d_3\mu_1}{2} \int_\Omega \vert\Delta w\vert^2+C_{12},
\end{align*}
where $C_i$ are positive constants.  Choose $\mu_1$ large such that $\frac{d_3\mu_1}{2}\geq C_*(C_2+C_4+1)$.  We apply the fact that $s\ln s\leq \epsilon s^2\ln s+C_\epsilon$, $\forall s>0$, and Young's inequality to obtain
\begin{equation}\label{216}
\frac{d}{dt}\Big(\int_\Omega u\ln u+ \int_\Omega v\ln v+ \frac{\mu_1}{2}\int_\Omega \vert\nabla w\vert^2\Big)+\int_\Omega u\ln u+ \int_\Omega v\ln v+ \frac{\mu_1}{2}\int_\Omega \vert\nabla w\vert^2\leq C,
\end{equation}
where $\epsilon=\min\{\frac{\alpha_1}{4},\frac{\alpha_2}{4}\}$.  Solving (\ref{216}) through Gronwall's lemma gives rise to (\ref{211}).
\end{proof}

\begin{lemma}\label{lemma25}
Under the same conditions as in Lemma \ref{lemma24}, there exists a constant $C>0$ such that
\begin{equation}\label{217}
\Vert u(\cdot,t) \Vert_{L^2}+\Vert v(\cdot,t) \Vert_{L^2}\leq C, \forall t \in(0,T_{\max}).
\end{equation}
\end{lemma}

\begin{proof}
We have from the integration by parts, H\"older's inequality and (\ref{23}) that
\begin{align}\label{218}
&\frac{1}{2}\frac{d}{dt}\int_\Omega u^2\nonumber  \\
=& -d_1\int_\Omega \vert\nabla u \vert^2-\chi\int_\Omega u\nabla\cdot(u\phi(w)\nabla w)+\int_\Omega (\alpha_1-\alpha_1u+\beta_1w)u^2 \nonumber  \\
=& -d_1\int_\Omega \vert\nabla u \vert^2-\frac{\chi}{2}\int_\Omega u^2 (\phi(w)\Delta w+\phi'(w)|\nabla w|^2)+\int_\Omega (\alpha_1-\alpha_1u+\beta_1w)u^2\nonumber \\
\leq & -d_1\int_\Omega \vert\nabla u \vert^2+\frac{\chi}{2}\Big( \Vert u\Vert^2_{L^3}\Vert  \Delta w\Vert_{L^3}+ \Vert u\Vert^2_{L^3}\Big\Vert  |\nabla w|^2\Big\Vert_{L^3}\Big)-\frac{\alpha_1}{2}\int_\Omega u^3+C_1,
\end{align}
where to derive the inequality we have assumed that $\Vert \phi(w)\Vert_{L^\infty}, \Vert\phi'(w)\Vert_{L^\infty}\leq 1 $ without loss of our generality.

To estimate (\ref{218}), we apply Lemma 3.5 of \cite{NSY} with $p=3$ to have that for any $\epsilon>0$, there exists a constant $C(\epsilon)>0$ such that $\Vert u\Vert_{L^3}\leq \epsilon \Vert \nabla u\Vert^\frac{2}{3}_{L^2}\Vert u\ln u\Vert^\frac{1}{3}_{L^1}+C(\epsilon)(\Vert u\ln u\Vert_{L^1}+\Vert u\Vert^\frac{1}{3}_{L^1})$.  This fact and (\ref{211}) imply that
\begin{equation}\label{219}
\Vert u\Vert^2_{L^3} \leq \big(\epsilon \Vert \nabla u\Vert^2_{L^2} +C_2\big)^\frac{2}{3}.
\end{equation}
Thanks to the Gagliardo--Nirenberg inequality \cite{TaoW} and (\ref{211}), we have
\begin{equation}\label{220}
\Vert \Delta w\Vert_{L^3} \leq C_3 \big(\Vert \nabla \Delta w\Vert^\frac{2}{3}_{L^2} \Vert\nabla w\Vert^\frac{1}{3}_{L^2}+\Vert \nabla w\Vert_{L^2}\big)\leq C_4 \big(\Vert \nabla \Delta w\Vert^\frac{2}{3}_{L^2}+1\big),
\end{equation}
and
\begin{align}\label{220b}
  \Big\Vert  |\nabla w|^2\Big\Vert_{L^3}=\Vert  \nabla w\Vert^2_{L^6}\leq C_5 \big(\Vert \nabla \Delta w\Vert^\frac{2}{3}_{L^2} \Vert w\Vert^\frac{4}{3}_{L^\infty}+\Vert  w\Vert^2_{L^\infty}\big)\leq C_6 \big(\Vert \nabla \Delta w\Vert^\frac{2}{3}_{L^2}+1\big).
\end{align}
Combining (\ref{219}) with (\ref{220}) and (\ref{220b}), we have from Young's inequality that
\begin{align*}
\Vert u\Vert^2_{L^3} \Vert \Delta w\Vert_{L^3} & \leq  C_7 \big(\epsilon \Vert \nabla u\Vert^2_{L^2} +C_0\big)^\frac{2}{3}\big(\Vert \nabla \Delta w\Vert^\frac{2}{3}_{L^2}+1\big)\\
& \leq  C_8 \big(\epsilon \Vert \nabla u\Vert^\frac{4}{3}_{L^2} +C_0\big)\big(\Vert \nabla \Delta w\Vert^\frac{2}{3}_{L^2}+1\big)\\
& \leq \frac{d_1}{2\chi}\Vert \nabla u\Vert^2_{L^2}+\epsilon\Vert\nabla\Delta w\Vert^2_{L^2}+C_9
\end{align*}
and that
\begin{align*}
\Vert u\Vert^2_{L^3} \Big\Vert  |\nabla w|^2\Big\Vert_{L^3} \leq \frac{d_1}{2\chi}\Vert \nabla u\Vert^2_{L^2}+\epsilon\Vert\nabla\Delta w\Vert^2_{L^2}+C_{10},
\end{align*}
therefore (\ref{218}) gives us
\begin{align}\label{221}
\frac{1}{2}\frac{d}{dt}\int_\Omega u^2+\frac{d_1}{2}\int_\Omega \vert\nabla u \vert^2\leq 2\epsilon \int_\Omega \vert \nabla \Delta w\vert^2-\frac{\alpha_1}{2}\int_\Omega u^3+C_{11}.
\end{align}
By the same arguments, we can show that
\begin{align}\label{222}
\frac{1}{2}\frac{d}{dt}\int_\Omega v^2+\frac{d_2}{2}\int_\Omega \vert\nabla v \vert^2\leq 2\epsilon \int_\Omega\vert \nabla \Delta w\vert^2-\frac{\alpha_2}{2}\int_\Omega v^3+C_{12}.
\end{align}

On the other hand, we operate $\nabla$ to $w$--equation in (\ref{11}).  Then it follows from Young's inequality and integration by parts that
\begin{align}\label{223}
\frac{1}{2}&\frac{d}{dt} \int_\Omega \vert \Delta w\vert^2=\int_\Omega \Delta w \Delta w_t=-\int_\Omega\nabla \Delta w \cdot \nabla w_t \nonumber \\
=&-\int_\Omega \nabla \Delta w \cdot \nabla\big(d_3\Delta w+\alpha_3(1-w)w-\beta_{31}uw-\beta_{32}vw\big)\nonumber \\
=&-d_3\int_\Omega \vert\nabla\Delta w\vert^2+ \int_\Omega \nabla \Delta w \cdot ( \alpha_3\nabla w-2\alpha_3 w\nabla w -\beta_{31}u\nabla w)  \nonumber\\
& + \int_\Omega\nabla\Delta w\cdot(-\beta_{31}w\nabla u -\beta_{32}v\nabla w-\beta_{32}w\nabla v) \nonumber \\
\leq& -\frac{d_3}{2}\int_\Omega \vert\nabla\Delta w\vert^2-\alpha_3\int_\Omega|\Delta w|^2+C_{13}\int_\Omega \vert\nabla u \vert^2+C_{14}\int_\Omega \vert\nabla v \vert^2+C_{15}\int_\Omega \vert\nabla w \vert^2\nonumber\\
&+C_{16}\int_\Omega u^2|\nabla w|^2+C_{17}\int_\Omega v^2|\nabla w|^2 \nonumber \\
 \leq& (-\frac{d_3}{2}+2\epsilon)\int_\Omega \vert\nabla\Delta w\vert^2-\alpha_3\int_\Omega|\Delta w|^2+2C_{13}\int_\Omega \vert\nabla u \vert^2+2C_{14}\int_\Omega \vert\nabla v \vert^2+C_{18}, \nonumber \\
\end{align}
where we used the inequalities
\[C_{16}\int_{\Omega}u^2|\nabla w|^2\leq C_{16} \Vert u\Vert^2_{L^3} \Big\Vert  |\nabla w|^2\Big\Vert_{L^3} \leq C_{13}\Vert \nabla u\Vert^2_{L^2}+\epsilon\Vert\nabla\Delta w\Vert^2_{L^2}+C_{19}\]
and
\[C_{17}\int_{\Omega}v^2|\nabla w|^2\leq C_{16}\Vert v\Vert^2_{L^3} \Big\Vert  |\nabla w|^2\Big\Vert_{L^3} \leq C_{14}\Vert \nabla v\Vert^2_{L^2}+\epsilon\Vert\nabla\Delta w\Vert^2_{L^2}+C_{20}.\]
Multiplying (\ref{223}) by $\mu_2$ sufficiently small with $\epsilon=\frac{d_3\mu_2}{4(2+\mu_2)}$, we add it with (\ref{221}) and (\ref{222}) to have
\begin{align}\label{224}
\frac{d}{dt}\Big(\frac{1}{2}\int_\Omega u^2+\frac{1}{2}\int_\Omega v^2+\frac{\mu_2}{2}\int_\Omega \vert \Delta w\vert^2\Big)+(\int_{\Omega}u^2+\int_{\Omega}u^2+\alpha_3\mu_2\int_\Omega|\Delta w|^2)\leq C.
\end{align}
Therefore (\ref{217}) follows from (\ref{224}) thanks to Gr\"onwall's inequality.
\end{proof}

\subsection{Existence and boundedness of global solutions}
We now present our main results on the global existence and uniform boundedness of positive classical solutions to (\ref{11}).
\begin{theorem}\label{theorem26}
Let $\Omega$ be a bounded domain in $\mathbb R^N$, $N\leq2$.  Suppose that all the parameters in (\ref{11}) are positive and $\phi\in C^3(\mathbb R;\mathbb R)$.  Then for any nonnegative initial data $(u_0,v_0,w_0)\in C^1(\bar \Omega)\times C^1(\bar \Omega)\times W^{2,p}(\Omega)$, $p>N$, there exists a unique triple $(u,v,w)$ of nonnegative bounded funcstions in $C^0(\bar \Omega\times[0,\infty))\cap C^{2,1}(\bar \Omega\times(0,\infty))$ which solve (\ref{11}) classically in $\Omega\times(0,\infty)$.
\end{theorem}

\begin{proof}
According to Theorem \ref{theorem21} and Lemma \ref{23}, we only need to show that $\Vert (u(\cdot,t),v(\cdot,t))\Vert_{L^\infty}$ is uniformly bounded for all $t\in(0,T_{\max})$, then $T_{\max}=\infty$ and the existence part of Theorem \ref{theorem26} follows.  Moreover, one can apply the standard parabolic boundary $L^p$--estimates and Schauder estimates in \cite{LSU} to verify that $u_t,v_t,w_t$ and all spatial partial derivatives of $u$, $v$ and $w$ up to second order are bounded on $\bar \Omega \times [0,\infty)$, therefore $(u,v,w)$ have the regularities stated in Theorem \ref{theorem26}

Thanks to Lemma \ref{lemma23} and Lemma \ref{lemma25}, we have that $\Vert\nabla w(\cdot,t)\Vert_{L^p}$ is bounded for any $p\in(1,\infty)$, then the uniform boundedness of $\Vert (u,v)(\cdot,t)\Vert_{L^\infty}$ follows from Proposition \ref{proposition21}.  Together with (\ref{23}), we conclude the proof of Theorem \ref{theorem26}.
\end{proof}

Our proof of the boundedness of $\Vert (u,v) \Vert_{L^\infty}$ hence the global existence of (\ref{11}) is restricted to the domain with dimension $N=1,2$ due to technical reasons.  It is well known that, in order to prove the boundedness of $\Vert \nabla w\Vert_{L^\infty}$, it is sufficient to prove the $L^p$--boundedness of $u$ and $v$ for some $p>2(N+1)$ for a wide class of reaction--diffusion system with biological relevance, for instance the Keller--Segel chemotaxis models with no cellular growth.  The logistic decay in (\ref{11}) contributes so one only needs to prove $L^p$--boundedness of $u$ and $v$ for some $p>N$.  Whether or not the logistic dynamics are sufficient to prevent finite or infinite time blowups for (\ref{11}) over higher dimensional domains is unclear in the literature and it is beyond the scope of this paper.  However, we can show the global existence and boundedness for the following parabolic--parabolic--elliptic system of (\ref{11}) in the following system regardless of space dimensions
\begin{equation}\label{225}
\left\{
\begin{array}{ll}
u_t=\nabla \cdot (d_1\nabla u -\chi u \phi(w) \nabla w)+\alpha_1(1-u)u+\beta_1uw,&x \in \Omega,t>0,\\
v_t=\nabla \cdot (d_2\nabla v -\xi v \phi(w) \nabla w)+\alpha_2(1-v)v+\beta_2vw,&x \in \Omega,t>0,\\
0=d_3\Delta w+\alpha_3(1-w)w-\beta_{31}uw-\beta_{32}vw,&   x \in \Omega,t>0,      \\
\partial_\textbf{n} u =\partial_\textbf{n} v =\partial_\textbf{n} w =0,&x\in \partial \Omega,t>0,\\
u(x,0)=u_0(x),v(x,0)=v_0(x),w(x,0)=w_0(x),&x\in\Omega,
\end{array}
\right.
\end{equation}
To be precise, we have the following results.
\begin{theorem}\label{theorem27}
Let $\Omega$ be a bounded domain in $\mathbb R^N$, $N\ge3$ with smooth boundary $\partial \Omega$.  Suppose that $\phi'(w)\geq0$, $\forall w\geq0$.  Then (\ref{225}) has a unique classical solution $(u,v,w)$ which exists globally in time and satisfies the following estimate with a positive constant $C$
\[\Vert u(\cdot,t)\Vert_{L^\infty}+\Vert v(\cdot,t)\Vert_{L^\infty}+\Vert w(\cdot,t)\Vert_{L^\infty}\leq C, \forall t\in(0,\infty).\]
\end{theorem}
\begin{proof}
By the same analysis as in the proof of Proposition \ref{proposition21}, we shall only need to show the uniform boundedness of $\Vert u(\cdot,t)\Vert_{L^\infty}$ and $\Vert v(\cdot,t)\Vert_{L^\infty}$.  Similar to (\ref{28}), we can obtain from the integration by parts
\begin{align}\label{227}
&\frac{1}{p}\frac{d}{dt}\int_{\Omega} u^p\nonumber \\
=&-\frac{4d_1(p-1)}{p^2} \int_{\Omega} \vert\nabla u^{\frac{p}{2}} \vert^2+\frac{\chi(p-1)}{p} \int_\Omega \nabla u^p\cdot (\phi(w)\nabla w)+\int_\Omega (\alpha_1-\alpha_1u-\beta_1w)u^p\nonumber \\
=&-\frac{4d_1(p-1)}{p^2} \int_{\Omega} \vert\nabla u^{\frac{p}{2}} \vert^2-\frac{\chi(p-1)}{p} \int_\Omega u^p\Big(\phi'(w) |\nabla w|^2+\phi(w)\Delta w\Big) \nonumber\\
&+\int_\Omega (\alpha_1-\alpha_1u-\beta_1w)u^p\nonumber \\
\leq &-\frac{4d_1(p-1)}{p^2} \int_{\Omega} \vert\nabla u^{\frac{p}{2}} \vert^2-\frac{\chi(p-1)}{p} \int_\Omega u^p \phi(w)\Delta w+\int_\Omega (\alpha_1-\alpha_1u-\beta_1w)u^p.
\end{align}
In light of the pointwise identity $d_3\Delta w+\alpha_3(1-w)w-\beta_{31}uw-\beta_{32}vw=0$, we have
\begin{align}\label{228}
&\frac{1}{p}\frac{d}{dt}\int_{\Omega} u^p \\ \nonumber
\leq &-\frac{4d_1(p-1)}{p^2} \int_{\Omega} \vert\nabla u^{\frac{p}{2}} \vert^2-\frac{\chi(p-1)}{d_3p} \int_\Omega u^p \phi(w)(\alpha_3w(w-1)-\beta_{31}uw-\beta_{32}vw)\nonumber \\
&+\int_\Omega (\alpha_1-\alpha_1u-\beta_1w)u^p\nonumber \\
\leq&-\frac{4d_1(p-1)}{p^2} \int_{\Omega} \vert\nabla u^{\frac{p}{2}} \vert^2-C_1\int_\Omega u^{p+1}+C_2,
\end{align}
where $C_1$ and $C_2$ are positive constants that depend on the system parameters and $\Vert w \Vert_{L^\infty}$.  Applying the Moser--iteration gives rise to the boundedness of $\Vert u(\cdot,t)\Vert_{L^\infty}$.  By the same way we can prove the boundedness of $\Vert v(\cdot,t)\Vert_{L^\infty}$.  The proof of this Theorem completes.
\end{proof}
The assumption $\phi'(w)\geq0$ corresponds to the biological situation that the intensity of the directed dispersals of predator species increases as prey density increases, therefore there is no group defense in prey species.  We have to leave it for future works on the global existence of (\ref{225}) and its fully parabolic counter--part over higher--dimensions.

\section{Linearized stability of homogeneous equilibrium}\label{section3}
From the viewpoint of mathematical modeling, it is interesting and meaningful to investigate (\ref{11}) for its nontrivial solutions describing the spatial distributions of the interacting species over the habitat.  We shall see from the mathematical analysis that the formation of such nontrivial patterns is due to the joint effect of prey--taxis and sensitivity function.

To illustrate the effect of prey--taxis on the pattern formation and for the simplicity of calculations, we restrict our attention to (\ref{11}) over one dimension and consider the following system
\begin{equation}\label{31}
\left\{
\begin{array}{ll}
u_t=(d_1 u' -\chi u\phi(w)  w')'+\alpha_1(1-u)u+\beta_1uw,&x \in (0,L), t>0,\\
v_t=(d_2 v' -\xi v\phi(w)  w')'+\alpha_2(1-v)v+\beta_2vw,&x \in (0,L),t>0,\\
w_t=d_3 w''+\alpha_3(1-w)w-\beta_{31}uw-\beta_{32}vw,&   x \in (0,L),t>0,\\
u'(x,t)=v'(x,t)=w'(x,t)=0,&x=0,L,t>0,\\
\end{array}
\right.
\end{equation}
where $'$ denotes the derivative taken with respect to $x$ and all the parameters in (\ref{31}) are the same as those in (\ref{11}).

We can find that (\ref{31}) has six equilibria and one of them is
\[(\bar u,\bar v,\bar w)=\Big(1+\frac{\beta_1}{\alpha_1}\bar w,1+\frac{\beta_2}{\alpha_2}\bar w,\bar w\Big), \bar w= \frac{\alpha_3-\beta_{31}-\beta_{32}}{\alpha_3+\frac{\beta_1\beta_{31}}{\alpha_1}+\frac{\beta_2\beta_{32}}{\alpha_2}},\]
which is positive if and only if $ \alpha_3>\beta_{31}+\beta_{32}$.  To explore the existence and stability of stationary and oscillatory nonconstant positive solutions to (\ref{31}), our starting point is the linear stability of $(\bar u,\bar v,\bar w)$ and we shall assume that it is positive from now on.

Linearizing (\ref{31}) around the constant equilibrium $(\bar{u},\bar{v},\bar{w})$ and letting $(u,v,w)=(\bar{u},\bar{v},\bar{w})+(U,V,W)$, $U,V,W$ being small perturbations from $(\bar{u},\bar{v},\bar{w})$, we obtain the following system of $(U,V,W)$
\begin{align*}
\begin{cases}
U_t \approx (d_1 U' -\chi \bar{u} \phi(\bar{w})  W')'-{\alpha_1}\bar{u}U+\beta_1 \bar{u}W, &x\in(0,L),t>0,\\
V_t \approx (d_2 V' -\xi \bar{v} \phi(\bar{w})  W')'-{\alpha_2}\bar{v}V+\beta_2 \bar{v}W,&x\in(0,L),t>0,\\
W_t \approx d_3 W''-{\alpha_3}\bar{w} W-\beta_{31} \bar{w}U -\beta_{32} \bar{w}V,    &x\in(0,L),t>0,\\
U'(x)=V'(x)=W'(x)=0, &x=0,L,t>0.
\end{cases}
\end{align*}
According to the standard linearized stability principle (\cite{Si} e.g.), the stability of $(\bar{u},\bar{v},\bar{w})$ is determined by eigenvalues to the following matrix
\begin{align}\label{32}
\begin{pmatrix}
-d_1 (\frac{k\pi}{L})^2-{\alpha_1}\bar{u} & 0 & \chi \bar{u} \phi(\bar{w}) (\frac{k\pi}{L})^2+\beta_1 \bar{u}\\
0 & -d_2 (\frac{k\pi}{L})^2-{\alpha_2}\bar{v} & \xi \bar{v} \phi(\bar{w}) (\frac{k\pi}{L})^2+\beta_2 \bar{v}\\
-\beta_{31}\bar{w} & -\beta_{32}\bar{w} & -d_3 (\frac{k\pi}{L})^2-{\alpha_3}\bar{w}
\end{pmatrix}.
\end{align}
We have the following result.
\begin{proposition}\label{proposition31}
Assume that all the parameters in (\ref{31}) are positive and the condition $\alpha_3>\beta_{31}+\beta_{32}$ is satisfied.  Suppose that $\phi(\bar w)<0$, then the positive equilibrium $(\bar{u},\bar{v},\bar{w})$ is locally asymptotically stable if $\chi < \chi_0$ and it is unstable if $\chi >\chi_0$, where
\begin{equation}\label{33}
\chi_0=\min_{k\in\mathbb{N}^+}{\{\chi^S_k,\chi^H_k\}}
\end{equation}
with
\begin{equation}\label{34}
\chi^S_k=-\frac{\beta_{32}\bar{v}H_1}{\beta_{31}\bar{u}H_2}\xi-\frac{H_1H_2H_3+H_2\beta_{31}\beta_1\bar{u}\bar{w}+H_1\beta_{32}\beta_2\bar{v}\bar{w}}{H_2\beta_{31}\bar{u}\bar{w}\phi(\bar{w})(\frac{k\pi}{L})^2},
\end{equation}
and
\begin{align}\label{35}
\chi^H_k=&-\frac{H_1^2H_2+H_1H_2^2+H_1^2H_3+H_1H_3^2+H_2^2H_3+H_2H_3^2+2H_1H_2H_3}{(H_1+H_3)\beta_{31}\bar{u}\bar{w}\phi(\bar{w})(\frac{k\pi}{L})^2}  \nonumber \\
&-\frac{(H_1+H_3)\beta_{31}\beta_1\bar{u}+(H_2+H_3)\beta_{32}\beta_2\bar{v}}{(H_1+H_3)\beta_{31}\bar{u}\phi(\bar{w})(\frac{k\pi}{L})^2}
-\frac{(H_2+H_3)\beta_{32}\bar{v}}{(H_1+H_3)\beta_{31}\bar{u}}\xi,
\end{align}
where
\[H_1=d_1(\frac{k\pi}{L})^2+{\alpha_1}\bar{u}, H_2=d_2(\frac{k\pi}{L})^2+{\alpha_2}\bar{v},H_3=d_3(\frac{k\pi}{L})^2+{\alpha_3}\bar{w}.\]
\end{proposition}
\begin{proof}
The characteristic equation for an eigenvalue $\sigma$ of the stability matrix (\ref{32}) is
\begin{align*}
\sigma^3+\eta_2(\chi,k)\sigma^2+\eta_1(\chi,k)\sigma+\eta_0(\chi,k)=0,
\end{align*}
where
{\small\begin{align*}
\eta_2(\chi,k)=&(d_1+d_2+d_3)(\frac{k\pi}{L})^2+{\alpha_1}\bar{u}+{\alpha_2}\bar{v}+{\alpha_3}\bar{w}>0,\\
\eta_1(\chi,k)=&\Big(d_1(\frac{k\pi}{L})^2+{\alpha_1}\bar{u} \Big)\Big(d_2(\frac{k\pi}{L})^2+{\alpha_2}\bar{v} \Big)+\Big(d_1(\frac{k\pi}{L})^2+{\alpha_1}\bar{u} \Big)\Big(d_3(\frac{k\pi}{L})^2+{\alpha_3}\bar{w} \Big)\\
&+\Big(d_2(\frac{k\pi}{L})^2+{\alpha_2}\bar{v} \Big)\Big(d_3(\frac{k\pi}{L})^2+{\alpha_3}\bar{w} \Big)+(\beta_{31}\beta_1 \bar{u}+\beta_{32}\beta_2 \bar{v})\bar{w}\\
&+(\beta_{31}\bar{u}\chi+\beta_{32}\bar{v}\xi)\bar{w}\phi(\bar{w})(\frac{k\pi}{L})^2,
\end{align*}}and
\begin{align*}
\eta_0(\chi,k)=&\Big(d_1(\frac{k\pi}{L})^2+{\alpha_1}\bar{u} \Big)\Big(d_2(\frac{k\pi}{L})^2+{\alpha_2}\bar{v} \Big)\Big(d_3(\frac{k\pi}{L})^2+{\alpha_3}\bar{w} \Big)+\Big(d_2(\frac{k\pi}{L})^2\\
&+{\alpha_2}\bar{v} \Big)\beta_{31}\beta_1\bar{u}\bar{w}+\Big(d_1(\frac{k\pi}{L})^2+{\alpha_1}\bar{u} \Big)\beta_2\beta_{32}\bar{v}\bar{w}+\Big(d_2(\frac{k\pi}{L})^2+{\alpha_2}\bar{v} \Big)\\
&\beta_{31}\bar{u}\bar{w}\phi(\bar{w})(\frac{k\pi}{L})^2\chi+\Big(d_1(\frac{k\pi}{L})^2+{\alpha_1}\bar{u} \Big)\beta_{32}\bar{v}\bar{w}\phi(\bar{w})(\frac{k\pi}{L})^2\xi.
\end{align*}
By the principle of the linearized stability (Theorem 5.2 in \cite{Si} e.g.), $(\bar{u},\bar{v},\bar{w})$ is asymptotically stable with respect to (\ref{31}) if and only if all eigenvalues of the matrix (\ref{32}) have negative real part, then according to the Routh--Hurwitz conditions, or Corollary 2.2 in \cite{LSW}, the constant equilibrium $(\bar{u},\bar{v},\bar{w})$ is locally asymptotically stable with respect to (\ref{11}) if and only if the following conditions hold for each $k\in\mathbb{N}^+$
\begin{align*}
\eta_0(\chi,k)>0, \eta_1(\chi,k)>0,~\text{and}~ \eta_1(\chi,k)\eta_2(\chi,k)-\eta_0(\chi,k)>0,
\end{align*}
while $(\bar{u},\bar{v},\bar{w})$ is unstable if one of the conditions above fails for some $k\in\mathbb{N}^+$.  Note that we always have that $\eta_2(\chi,k)>0$ for each $k\in\mathbb N^+$; moreover, $\eta_1(\chi,k)>0$ if $\eta_0(\chi,k)>0$ and $\eta_1(\chi,k)\eta_2(\chi,k)-\eta_0(\chi,k)>0$, therefore $(\bar{u},\bar{v},\bar{w})$ is unstable if there exists $k\in\mathbb{N}^+$ such that either $\eta_0(\chi,k)<0$ or $\eta_1(\chi,k)\eta_2(\chi,k)-\eta_0(\chi,k)<0$.  On the other hand, since $\phi(\bar w)<0$, it follows from straightforward calculations that
\[\eta_0(\chi,k)<0 \text{ if and only if }\chi>\chi^S_k,\] and
\[\eta_1(\chi,k)\eta_2(\chi,k)-\eta_0(\chi,k)<0 \text{ if and only if }\chi>\chi^H_k,\]
therefore the constant solution $(\bar{u},\bar{v},\bar{w})$ is unstable if there exists $k\in\mathbb N^+$ such that $\chi>\chi^S_k$ or $\chi>\chi^H_k$, i.e., if $\chi$ is larger than the minimum of $\chi^S_k$ and $\chi^H_k$ over $k\in\mathbb N$.  Similarly we can show that $(\bar{u},\bar{v},\bar{w})$ is locally asymptotically stable if $\chi<\chi_0$.  This finishes the proof.
\end{proof}
According to Proposition \ref{proposition31}, $(\bar{u},\bar{v},\bar{w})$ becomes unstable as $\chi$ surpasses the threshold value $\chi_0$ if $\phi(\bar{w})<0$ which we shall assume from now on.  We would like to point out that biologically $\phi(\bar{w})<0$ describes situation that a huge amount of preys can aggregate to form group defense and keep predators away from the habitat when the prey population density surpasses $\bar w$.  This amounts to a switch from prey--attraction to prey--repulsion in the predation.  See \cite{AM,Hamilton,GMJW,XR} and the references therein for the detailed description of prey group defense.

We will see in our coming mathematical analysis and numerical simulations that $(\bar u,\bar v,\bar w)$ loses its stability to time--periodic patterns through Hopf bifurcation if $\chi=\chi^H_k$ and to stationary patterns through steady state bifurcation if $\chi=\chi^S_k$.  Therefore we have used the indices $H$ and $S$ to denote Hopf and steady state bifurcation respectively.  Here and in the rest part of this paper, we study the effect of prey--taxis on the formation of nontrivial patterns to (\ref{31}).  Without losing the generality of our analysis, we treat $\chi$ as the variable parameter and fix all the rest parameters, while similarly, we have that Proposition \ref{proposition31} also holds for $\xi>\xi_0=\min_{k\in\mathbb N^+}\{\xi_k^H,\xi_k^S\}$, where $\xi_k^H$ and $\xi_k^S$ are functions of $\chi$.
\begin{corollary}\label{Corollary1}
Suppose that $\phi(\bar w)>0$ and all the rest conditions in Proposition \ref{proposition31} are satisfied, the homogeneous equilibrium $(\bar u,\bar v,\bar w)$ is locally asymptotically stable if $\chi>\bar \chi_0$ and it is unstable if $\chi<\bar \chi_0$, where
\[\bar \chi_0=\max_{k\in\mathbb N}\{\chi^S_k,\chi^H_k\}.\]
\end{corollary}
When $\phi(\bar w)>0$, both $\chi^S_k$ and $\chi^H_k$ are negative for any $k\in\mathbb N^+$ hence $\bar \chi_0<0$, therefore $(\bar u,\bar v,\bar w)$ is always stable according to Corollary \ref{Corollary1} since $\chi>0$.  This result corresponds to the widely held belief (\cite{LHL2} e.g.) that prey--taxis stabilizes homogeneous equilibrium and inhibits the formation of spatial patterns for one--predator and one--prey system.  It states that the same holds true for two--predator and one--prey model as long as $\phi(\bar w)>0$.  However, if $\phi(\bar w)<0$, the prey--taxis destabilizes homogeneous equilibrium which becomes unstable as $\chi$ surpasses $\chi_0$.  Therefore, in order to investigate the formation of nontrivial patterns in (\ref{31}), we shall assume that $\phi(\bar w)<0$ in the coming analysis.

\begin{remark}\label{remark31}
As we shall see in our coming analysis, the occurrence of Hopf or steady state bifurcation at $(\bar u,\bar v,\bar w)$ depends on whether $\chi_{0}$ in (\ref{33}) is achieved at $\min_{k\in \mathbb{N}^+}\chi_{k}^{S}$ or $\min_{k\in \mathbb{N}^+}\chi_{k}^{H}$.  We divide our discussions into the following cases:
\begin{enumerate}[(1).]
\item If $\chi_{0}=\chi_{k_{0}}^{S}<\min_{k\in \mathbb{N}^+}\chi_{k}^{H}$,  then $\eta_{0}(\chi_0,k_0)=0$ and the eigenvalues of (\ref{32}) at $\chi_0$ are
\[\sigma_{1}^{S}(\chi_0,k_0)=0, \sigma_{2,3}^{S}(\chi_0,k_0)=\frac{-\eta_2(\chi_0,k_0)\pm\sqrt{\eta_2^2(\chi_0,k_0)-4\eta_1(\chi_0,k_0)}}{2}.\]
Since $(\bar{u},\bar{v},\bar{w})$ is unstable for all $\chi>\chi_{0}$, we know that (\ref{32}) has at least one eigenvalue with positive real part when $\chi=\chi^S_k, k\neq k_{0}$. This fact will be applied in the proof of the stability of steady state bifurcating solutions to (\ref{31}) around $(\bar{u}, \bar{v}, \bar{w}, \chi_{k}^{S})$.  In particular, it implies that the only stable bifurcating solutions must be on the branch around $(\bar{u}, \bar{v}, \bar{w}, \chi^S_{k_0})$ that turns to the right while all the rest branches are always unstable.  Moreover, Hopf bifurcation does not occur at $(\bar{u}, \bar{v}, \bar{w})$ in this case.
\item If $\chi_{0}=\chi_{k_{1}}^{H}<\min_{k\in \mathbb{N}^+}\chi_{k}^{S}$, then $\eta_{0}(\chi_0,k_1)=\eta_1(\chi_0,k_1)\eta_{2}(\chi_0,k_1)$ and the eigenvalues of (\ref{32}) at $\chi_0$ are
    \[\sigma_{1}^{H}(\chi_0,k_1)=-\eta_2(\chi_0,k_1)<0, \sigma_{2,3}^{H}(\chi_0,k_1)=\pm\sqrt{-\eta_1(\chi_0,k_1)}.\]
      Similarly we can show that the steady state bifurcating solutions around $(\bar{u}, \bar{v},\break \bar{w}, \chi_{k}^{S})$ are always unstable for all $k\in \mathbb{N}^+$.  Moreover we see that $\eta_1(\chi_0,k_{1})>0$ and (\ref{32}) has three eigenvalues: $\sigma_{1}^{H}(\chi_0,k_1)=-\eta_{2}(\chi_0,k_1), \sigma_{2,3}^{H}(\chi_0,k_1)=\pm\sqrt{\eta_1(\chi_0,k_1)}i$.  This indicates the possibility of a Hopf bifurcation and the emergence of time-periodic spatial patterns in (\ref{31}) when $\chi=\chi^H_{k_{1}}$. To prove this claim, we argue by contradiction and assume that $\eta_1(\chi_0,k_{1})<0$, therefore $\sigma_{2}^{H}(\chi_0,k_{1})=\sqrt{-\eta_1(\chi_0,k_1)}>0$ and Re$(\sigma_{2})>0$ if $\chi$ is sightly smaller than $\chi_{k_{1}}^{H}$, however, this indicates that $(\bar{u}, \bar{v},\bar{w})$ is unstable for $\chi<\chi_{k_{1}}^{H}=\chi_{0}$, which is a contradiction. Our numerical simulations in Section \ref{section6} support the existence of Hopf bifurcation and time--periodic patterns to (\ref{31}).
  \item If $\chi_{0}=\chi_{k_{0}}^{S}=\chi_{k_{1}}^{H}$. In this case, (\ref{32}) has three eigenvalues $\sigma_{1}=-\eta_2(\chi,k)<0, \sigma_{2}=\sigma_{3}=0$, and linear stability of the $(\bar{u}, \bar{v}, \bar{w})$ is lost since there are two zero eigenvalues.  This inhibits the application of our steady state and Hopf bifurcation analysis which requires the null space of (\ref{32}) to be one-dimensional. Therefore we assume that $\chi_{k}^{S}\neq\chi_{k}^{H},~\forall k\in\mathbb{N}^+$ in our bifurcation analysis.
\end{enumerate}
\end{remark}
In general, it is not obvious to determine when case (1) or case (2) in Remark \ref{remark31} occurs.  However, if the interval length $L$ is sufficiently small, we have
\[\chi^S_k\approx -\frac{d_1d_3}{\beta_{31}\bar u\bar w \phi(\bar w)}\Big(\frac{k\pi}{L}\Big)^2<-\frac{d_1d_3+d_2(d_1+d_3)+d_2^2}{\beta_{31}\bar u\bar w \phi(\bar w)}\Big(\frac{k\pi}{L}\Big)^2\approx\chi^H_k,k\in\mathbb N^+,\]
since $\phi(\bar w)<0$.  This implies that for small intervals, $\chi_0=\min_{k\in\mathbb N^+} \chi^S_k=\chi^S_1$ and $(\bar u,\bar v,\bar w)$ loses its stability to the steady state bifurcating solution as $\chi$ surpasses $\chi_0$.  Since $k_0=1$, the bifurcating solution has a stable wave mode $\cos \frac{\pi x}{L}$ which is monotone in $x$.  Moreover, we shall observe from our numerics that $k_0$ is increasing or non-decreasing in $L$.  Therefore, small domain only supports monotone stable solutions, while large interval supports non-monotone solutions, at least when $\chi$ is around $\chi_0$.  Actually, if $(u(x),v(x),w(x))$ is an increasing solution to (\ref{31}), $(u(L-x),v(L-x),w(L-x))$ is a decreasing solution, then one can construct non-monotone solutions to (\ref{31}) over $(0,2L)$, $(0,3L)$,...by reflecting and periodically extending the monotone ones at $x=L$, $2L$, $3L$,...

\section{Nonconstant positive steady states}\label{section4}
This section is devoted to studying nonconstant positive steady states to system (\ref{31}), i.e., nonconstant solutions to the following system
\begin{equation}\label{41}
\left\{
\begin{array}{ll}
(d_1 u' -\chi u\phi(w)  w')'+\alpha_1(1-u)u+\beta_1uw=0,&x \in (0,L),\\
(d_2 v' -\xi v\phi(w)  w')'+\alpha_2(1-v)v+\beta_2vw=0,&x \in (0,L),\\
d_3 w''+\alpha_3(1-w)w-\beta_{31}uw-\beta_{32}vw=0,&   x \in (0,L),      \\
u'(x)=v'(x)=w'(x)=0,&x=0,L,\\
\end{array}
\right.
\end{equation}
where $u$, $v$ and $w$ are functions of $x$ and all the parameters are the same as those in (\ref{31}).  We assume that $\alpha_3>\beta_{31}+\beta_{32}$ so that $(\bar u, \bar v,\bar w)$ is the unique positive equilibrium to (\ref{41}).

In order to look for nonconstant positive solutions to (\ref{41}), we shall perform steady state bifurcation analysis at $(\bar u, \bar v,\bar w)$.  When $\phi(\bar w)<0$, we already know that prey--taxis $\chi$ destabilizes $(\bar{u},\bar{v},\bar{w})$ which becomes unstable when $\chi$ surpasses $\chi_0$, therefore we are concerned with the conditions under which the spatially inhomogeneous solutions emerge through bifurcation as $\chi$ increases.  We refer these as prey--taxis induced patterns in analogy to Turing's instability.

\subsection{Steady state bifurcation}
To apply the bifurcation theory of Crandall--Rabinowitz \cite{CR,CR2} with $\chi$ being the bifurcation parameter, we introduce the spaces
\[\mathcal{X}=\{w\in H^2(0,L)\vert w'(0)=w'(L)=0\}, \mathcal{Y}=L^2(0,L)\]
and convert (\ref{41}) into the following abstract equation
\[\mathcal{F}(u,v,w,\chi)=0, ~(u,v,w,\chi)\in \mathcal{X}\times\mathcal{X}\times\mathcal{X}\times\mathbb{R},\]
where
\begin{align}\label{42}
\mathcal{F}(u,v,w,\chi)=
\begin{pmatrix}
(d_1 u' -\chi u\phi(w)  w')'+\alpha_1(1-u)u+\beta_1uw\\
(d_2 v' -\xi v\phi(w)  w')'+\alpha_2(1-v)v+\beta_2vw\\
d_3 w''+\alpha_3(1-w)w-\beta_{31}uw-\beta_{32}vw
\end{pmatrix}.
\end{align}
It is easy to see that $\mathcal{F}(\bar{u},\bar{v},\bar{w},\chi)=0$ for any $\chi\in\mathbb{R}$ and $\mathcal{F}:\mathcal{X}\times\mathcal{X}\times\mathcal{X}\times\mathbb{R}\rightarrow \mathcal{Y}\times\mathcal{Y}\times\mathcal{Y}$ is analytic.  Moreover, for any fixed $(\hat u,\hat v,\hat w)\in\mathcal{X}\times\mathcal{X}\times\mathcal{X}$, the Fr\'{e}chet derivative of $\mathcal F$ is given by
\begin{align}\label{43}
&D\mathcal{F}(\hat u,\hat v,\hat w,\chi)(u,v,w) \nonumber\\
=&
\begin{pmatrix}
d_1 u''-\chi(u\phi(\hat w)\hat w' + \hat uw\phi'(\hat w)\hat w' + \hat u\phi(\hat w)w')'+D\mathcal{F}_1\\
d_2 v''-\xi(v\phi(\hat w)\hat w' + \hat vw\phi'(\hat w)\hat w' + \hat v\phi(\hat w)w')'+D\mathcal{F}_2\\
d_3 w''-\beta_{31}\hat wu-\beta_{32}\hat wv+(\alpha_3-2\alpha_3 \hat w-\beta_{31}\hat u-\beta_{32}\hat v)w
\end{pmatrix},
\end{align}
where $D\mathcal{F}_1=(\alpha_1-2\alpha_1 \hat u+\beta_1 \hat w)u+\beta_1 \hat u w $ and $D\mathcal{F}_2=(\alpha_2-2\alpha_2 \hat v+\beta_2 \hat w)v+\beta_2 \hat v w $.  We collect the following facts about $\mathcal F$.
\begin{lemma}\label{lemma41}
$D\mathcal{F}(\hat u,\hat v,\hat w,\chi)(u,v,w):\mathcal{X}\times\mathcal{X}\times\mathcal{X}\times\mathbb{R}\rightarrow \mathcal{Y}\times\mathcal{Y}\times\mathcal{Y}$ is a Fredholm operator with zero index.
\end{lemma}
\begin{proof}
We denote $\mathbf{u}=(u,v,w)^T$ and rewrite (\ref{43}) as
\begin{align*}
D\mathcal{F}(\hat u,\hat v,\hat w,\chi)(u,v,w)=\mathcal{A}_0(\mathbf{u})\mathbf{u}''+\mathbf{F}_0(x,\mathbf{u},\mathbf{u}'),
\end{align*}
where
\begin{align*}
\mathcal{A}_0=
\begin{pmatrix}
d_1 & 0 & -\chi \hat u\phi(\hat w)\\
0 & d_2 & -\xi \hat v\phi(\hat w)\\
0 & 0 & d_3
\end{pmatrix}
\end{align*}
and
\begin{align*}
\mathbf{F}_0(x,\mathbf{u},\mathbf{u}')=
\begin{pmatrix}
-\chi(u\phi(\hat w)\hat w' + \hat uw\phi'(\hat w)\hat w')' + (\hat u\phi(\hat w))'w'+D\mathcal{F}_1\\
-\xi(v\phi(\hat w)\hat w' + \hat vw\phi'(\hat w)\hat w')' + (\hat v\phi(\hat w))'w'+D\mathcal{F}_2\\
-\beta_{31}\hat wu-\beta_{32}\hat wv+(\alpha_3-2\alpha_3 \hat w-\beta_{31}\hat u-\beta_{32}\hat v)w
\end{pmatrix},
\end{align*}
therefore operator (\ref{43}) is elliptic since all eigenvalues of $\mathcal{A}_0$ are positive. According to Remark 2.5 (case 2) in Shi and Wang \cite{SW} with $N=1$, $\mathcal{A}_0$ satisfies the Agmon's condition (see Theorem 4.4 of \cite{Am1} and Definition 2.4 in \cite{SW}).  Therefore, $D\mathcal{F}(\hat u,\hat v,\hat w,\chi)(u,v,w)$ is Fredholm with zero index due to Theorem 3.3 and Remark 3.4 of \cite{SW}.
\end{proof}

To seek non--trivial solutions of (\ref{41}) that bifurcate from equilibrium $(\bar{u},\bar{v},\bar{w})$, we first check the following necessary condition
\begin{align}\label{44}
\mathcal{N}(D\mathcal{F}(\bar{u},\bar{v},\bar{w},\chi) )\not= \{\textbf{0}\},
\end{align}
where $\mathcal N$ denotes the null space.  Taking $(\hat u, \hat v, \hat w)=(\bar{u},\bar{v},\bar{w})$ in (\ref{43}), we see that the null space in (\ref{44}) consists of solutions to the following problem
\begin{align}\label{45}
\begin{split}
\begin{cases}
d_1u''-\chi \bar{u}\phi(\bar{w})w''-\alpha_1 \bar{u} u+\beta_1\bar{u}w=0, &x\in(0,L),\\
d_2v''-\xi \bar{v}\phi(\bar{w})w''-\alpha_2 \bar{v} v+\beta_2\bar{v}w=0, &x\in(0,L),\\
d_3 w''-\beta_{31}\bar{w} u-\beta_{32}\bar{w} v-\alpha_3 \bar{w} w=0, &x\in(0,L),\\
u'(x)=v'(x)=w'(x)=0,&x=0,L.\\
\end{cases}
\end{split}
\end{align}
In order to verify (\ref{44}), we substitute the following eigen--expansions into (\ref{45})
\[u(x)=\sum_{k=0}^{\infty}t_k \cos{\frac{k\pi x}{L}}, v(x)=\sum_{k=0}^{\infty}s_k \cos{\frac{k\pi x}{L}}, w(x)=\sum_{k=0}^{\infty}r_k \cos{\frac{k\pi x}{L}},\]
$t_k$, $s_k$ and $r_k$ constants and collect
\begin{align}\label{46}
\begin{pmatrix}
-d_1 (\frac{k\pi}{L})^2-{\alpha_1}\bar{u} & 0 & \chi \bar{u} \phi(\bar{w}) (\frac{k\pi}{L})^2+\beta_1 \bar{u}\\
0 & -d_2 (\frac{k\pi}{L})^2-{\alpha_2}\bar{v} & \xi \bar{v} \phi(\bar{w}) (\frac{k\pi}{L})^2+\beta_2 \bar{v}\\
-\beta_{31}\bar{w} & -\beta_{32}\bar{w} & -d_3 (\frac{k\pi}{L})^2-{\alpha_3}\bar{w}
\end{pmatrix}
\begin{pmatrix}
t_k\\s_k\\r_k
\end{pmatrix}
=
\begin{pmatrix}
0\\0\\0
\end{pmatrix}.
\end{align}
$k=0$ can be easily ruled out since $\alpha_3>\beta_{31}+\beta_{32}$.  For each $k\in \mathbb N^+$, (\ref{45}) has nonzero solutions $(t_k,s_k,r_k)$ if and only if the coefficient matrix of (\ref{46}) is singular or equivalently
\begin{align}\label{47}
\chi=\chi^S_k=-\frac{\beta_{32}\bar{v}H_1}{\beta_{31}\bar{u}H_2}\xi-\frac{H_1H_2H_3+H_2\beta_{31}\beta_1\bar{u}\bar{w}+H_1\beta_{32}\beta_2\bar{v}\bar{w}}{H_2\beta_{31}\bar{u}\bar{w}\phi(\bar{w})(\frac{k\pi}{L})^2},
\end{align}
where $H_1,H_2$ and $H_3$ are given in Proposition \ref{proposition31}.  Note that $\chi_k^S$ in (\ref{47}) is the same as (\ref{34}). $\chi^S_k>0$ if and only if
\begin{align*}
\xi<-\frac{H_1H_2H_3+H_2\beta_{31}\beta_1\bar{u}\bar{w}+H_1\beta_{32}\beta_2\bar{v}\bar{w}}{H_1\beta_{32}\bar{v}\bar{w}\phi(\bar{w})(\frac{k\pi}{L})^2}.
\end{align*}
Condition (\ref{44}) is satisfied if $\chi=\chi^S_k$ and $\mathcal{N}(D\mathcal{F}(\bar{u},\bar{v},\bar{w},\chi^S_k) )=\text{span}{\{(\bar{u}_k,\bar{v}_k,\bar{w}_k) \}}$ which is of one dimension
\begin{align}\label{48}
\bar{u}_k=P_k \cos{\frac{k\pi x}{L}}, \bar{v}_k=Q_k \cos{\frac{k\pi x}{L}}, \bar{w}_k= \cos{\frac{k\pi x}{L}},
\end{align}
where
\begin{align}\label{49}
P_k=-\frac{\beta_{32}}{\beta_{31}}\frac{\xi \bar{v} \phi(\bar{w}) (\frac{k\pi}{L})^2+\beta_2 \bar{v}}{d_2 (\frac{k\pi}{L})^2+{\alpha_2}\bar{v}}-\frac{d_3 (\frac{k\pi}{L})^2+{\alpha_3}\bar{w}}{\beta_{31}\bar{w}},
\end{align}
and
\begin{align}\label{410}
Q_k=\frac{\xi \bar{v} \phi(\bar{w}) (\frac{k\pi}{L})^2+\beta_2 \bar{v}}{d_2 (\frac{k\pi}{L})^2+{\alpha_2}\bar{v}}.
\end{align}

Having the potential bifurcation value $\chi^S_k$ in (\ref{47}), we now prove in the following theorem that the steady state bifurcation occurs at $(\bar{u},\bar{v},\bar{w},\chi^S_k)$ for each $k\in\mathbb{N}^+$, which establishes existence of nonconstant positive solutions to (\ref{41}).
\begin{theorem}\label{theorem42}
Assume that $\alpha_3>\beta_{31}+\beta_{32}$ and $\phi(\bar{w})<0$.  Suppose that for positive integers $k,j\in\mathbb{N}^+$,
\begin{align}\label{411}
\chi^S_k\not=\chi^S_j, \forall k\not=j \text{~and~} \chi^S_k\not=\chi^H_k,\forall k\in\mathbb{N}^+,
\end{align}
where $\chi^S_k$ and $\chi^H_k$ are given by (\ref{34}) and (\ref{35}) respectively.  Then for each $k\in\mathbb{N}^+$, there exist a positive constant $\delta$ and a unique one--parameter curve $\Gamma_k(s)=\{(u_k(s,x),v_k(s,x),w_k(s,x),\chi_k(s)): s \in(-\delta,\delta)\}$ of spatially inhomogeneous solutions $(u,v,w,\chi)\in\mathcal{X}\times\mathcal{X}\times\mathcal{X}\times\mathbb R$ to (\ref{41}) that bifurcate from  $(\bar{u},\bar{v},\bar{w})$ at $\chi=\chi^S_k$.   Moreover, the solutions are smooth functions of $s$ such that
\begin{align}\label{412}
\chi_k(s)=\chi^S_k+O(s), s\in(-\delta,\delta),
\end{align}
and
\begin{align}\label{413}
(u_k(s,x),v_k(s,x),w_k(s,x))=(\bar{u},\bar{v},\bar{w})+s(\bar{u}_k,\bar{v}_k,\bar{w}_k)+\textbf{O}(s^2), s\in(-\delta,\delta),
\end{align}
where $(\bar{u}_k,\bar{v}_k,\bar{w}_k)$ is given by (\ref{48}) and $\textbf{O}(s^2)\in\mathcal{Z}$ is in the closed complement of $\mathcal{N}(D\mathcal{F}(\bar{u},\bar{v},\bar{w},\chi) )$ defined by
\begin{align}\label{414}
\mathcal{Z}=\Big\{(u,v,w)\in\mathcal{X}\times\mathcal{X}\times\mathcal{X}\Big\vert \int_{0}^{L}u\bar{u}_k+v\bar{v}_k+w\bar{w}_k dx=0 \Big\}.
\end{align}
\end{theorem}

\begin{proof}
All the necessary conditions except the following have been verified in order to apply the Crandall--Rabinowitz local theory in \cite{CR}
\begin{align}\label{415}
\frac{d}{d\chi}(D\mathcal{F}(\bar{u},\bar{v},\bar{w},\chi))(\bar{u}_k,\bar{v}_k,\bar{w}_k)\vert_{\chi=\chi^S_k}\not\in \mathcal{R} (D\mathcal{F}(\bar{u},\bar{v},\bar{w},\chi)),
\end{align}
where $ \mathcal{R} $ is the range of the operator.  We argue by contradiction and suppose that condition (\ref{415}) fails, then there exists a nontrivial solution $(u,v,w)$ that satisfies
\begin{align}\label{416}
\begin{split}
\begin{cases}
d_1u''-\chi^S_k\bar{u}\phi(\bar{w})w''-\alpha_1 \bar{u} u+\beta_1\bar{u}w=(\frac{k\pi}{L})^2\bar{u}\phi(\bar{w})\cos{\frac{k\pi x}{L}}, &x\in(0,L),\\
d_2v''-\xi \bar{v}\phi(\bar{w})w''-\alpha_2 \bar{v} v+\beta_2\bar{v}w=0, &x\in(0,L),\\
d_3 w''-\beta_{31}\bar{w} u-\beta_{32}\bar{w} v-\alpha_3 \bar{w} w=0, &x\in(0,L).\\
u'(x)=v'(x)=w'(x)=0,&x=0,L.\\
\end{cases}
\end{split}
\end{align}
Multiplying equations in (\ref{416}) by $\cos{\frac{k\pi x}{L}}$ and integrating them over $(0,L)$ by parts, we obtain that
\begin{align}\label{417}
&\begin{pmatrix}
-d_1 (\frac{k\pi}{L})^2-{\alpha_1}\bar{u} & 0 & \chi^S_k \bar{u} \phi(\bar{w}) (\frac{k\pi}{L})^2+\beta_1 \bar{u}\\
0 & -d_2 (\frac{k\pi}{L})^2-{\alpha_2}\bar{v} & \xi \bar{v} \phi(\bar{w}) (\frac{k\pi}{L})^2+\beta_2 \bar{v}\\
-\beta_{31}\bar{w} & -\beta_{32}\bar{w} & -d_3 (\frac{k\pi}{L})^2-{\alpha_3}\bar{w}
\end{pmatrix}
\begin{pmatrix}
\int_{0}^{L}u\cos{\frac{k\pi x}{L}}dx \\
\int_{0}^{L}v\cos{\frac{k\pi x}{L}}dx\\
\int_{0}^{L}w\cos{\frac{k\pi x}{L}}dx
\end{pmatrix}\nonumber\\
=&
\begin{pmatrix}
\frac{(k\pi)^2\bar{u}\phi(\bar{w})}{2L}\\
0\\0
\end{pmatrix}.
\end{align}
The coefficient matrix is singular because of (\ref{47}), then we reach a contradiction and this completes the proof of condition (\ref{415}).  Finally the statements in Theorem \ref{theorem42} follow from Theorem 1.7 of \cite{CR}.
\end{proof}

\subsection{Stability of bifurcating solutions near $(\bar{u},\bar{v},\bar{w},\chi^S_k)$}
Now we proceed to study the stability of the spatially inhomogeneous solution $(u_k(s,x), v_k(s,x),\break w_k(s,x))$ established in Theorem \ref{theorem42}. Here the stability or instability is that of the bifurcation solution regarded as an equilibrium of system (\ref{31}). To this end, we want to determine the turning direction of the bifurcation branch $\Gamma_k(s)$ around each bifurcation point $\chi_k^S$. It is easy to see that the operator $\mathcal{F}$ is $C^4$--smooth if $\phi$ is $C^5$--smooth, therefore according to Theorem 1.8 in \cite{CR2}, we can write the following expansions
\begin{align}\label{418}
\begin{split}
\begin{cases}
u_k(s,x)=\bar{u}+sP_k\cos{\frac{k\pi x}{L}}+s^2\varphi_1(x)+s^3\varphi_2(x)+o(s^3),\\
v_k(s,x)=\bar{v}+sQ_k\cos{\frac{k\pi x}{L}}+s^2\psi_1(x)+s^3\psi_2(x)+o(s^3),\\
w_k(s,x)=\bar{w}+s\cos{\frac{k\pi x}{L}}+s^2\gamma_1(x)+s^3\gamma_2(x)+o(s^3),\\
\chi_k(s)=\chi^S_k+s\mathcal{K}_1+s^2\mathcal{K}_2+o(s^2),
\end{cases}
\end{split}
\end{align}
where $(\varphi_i,\psi_i,\gamma_i)\in\mathcal{Z}$ in (\ref{414}) and $\mathcal{K}_i$ are constants for $i=1,2$. $o(s^3)$ are taken with respect to the $\mathcal{X}$--topology and $o(s^2)$ is a constant.  Moreover we have the Taylor expansion
\begin{align}\label{419}
\phi(w_k(s,x))=\phi(\bar{w})+s \phi'(\bar{w})\cos{\frac{k\pi x}{L}}+s^2\Big(\phi'(\bar{w})\gamma_1+\frac{1}{2}\phi''(\bar{w})\cos^2{\frac{k\pi x}{L}}  \Big)+o(s^2).
\end{align}
First of all, we claim that the bifurcation branch $\Gamma_k$ is of pitch--fork type by showing $\mathcal{K}_1=0$. Substituting (\ref{418}) into (\ref{41}) and collecting $s^2$ terms, we obtain the following system
{\small\begin{align}\label{420}
\begin{split}
\begin{cases}
d_1\varphi''_1-\chi^S_k\bar{u}\phi(\bar{w})\gamma''_1=({\alpha_1}\varphi_1-\beta_1\gamma_1)\bar{u}-\mathcal{K}_1\bar{u}\phi(\bar{w})(\frac{k\pi}{L})^2\cos{\frac{k\pi x}{L}}+R_k,&x\in(0,L),\\
d_2\psi''_1- \xi\bar{v}\phi(\bar{w})\gamma''_1=({\alpha_2}\psi_1-\beta_2\gamma_1)\bar{v}+S_k,&x\in(0,L),\\
d_3\gamma''_1=(\beta_{31}P_k+\beta_{32}Q_k+{\alpha_3})\cos^2{\frac{k\pi x}{L}}+(\beta_{31}\varphi_1+\beta_{32}\psi_1+{\alpha_3}\gamma_1)\bar{w},&x\in(0,L),\\
\varphi'_1(x)=\psi'_1(x)=\gamma'_1(x)=0,&x=0,L,
\end{cases}
\end{split}
\end{align}}where
\begin{align*}
R_k=-\chi^S_k(\frac{k\pi}{L})^2(\bar{u}\phi'(\bar{w})+P_k\phi(\bar{w}))\cos{\frac{2k\pi x}{L}}+({\alpha_1}P_k-\beta_1)P_k\cos^2{\frac{k\pi x}{L}},
\end{align*}
and
\begin{align*}
S_k=-\xi(\frac{k\pi}{L})^2(\bar{v}\phi'(\bar{w})+Q_k\phi(\bar{w}))\cos{\frac{2k\pi x}{L}}+({\alpha_2}Q_k-\beta_2)Q_k\cos^2{\frac{k\pi x}{L}}.
\end{align*}
Multiplying the equations in (\ref{420}) by $\cos{\frac{k\pi x}{L}}$ and integrating them over $(0,L)$ by parts give us
\begin{align}\label{421}
\frac{\bar{u}\phi(\bar{w})(k\pi)^2}{2L}\mathcal{K}_1=&\Big(d_1(\frac{k\pi}{L})^2+{\alpha_1}\bar{u}\Big)\int_{0}^{L}\varphi_1\cos{\frac{k\pi x}{L}}dx \nonumber\\
&+\Big(-\chi^S_k\bar{u}\phi(\bar{w})(\frac{k\pi}{L})^2-\beta_1\bar{u} \Big)   \int_{0}^{L}\gamma_1\cos{\frac{k\pi x}{L}}dx,
\end{align}
\begin{equation}\label{422}
\Big(d_2(\frac{k\pi}{L})^2+{\alpha_2}\bar{v})\Big)\int_{0}^{L}\psi_1\cos{\frac{k\pi x}{L}}dx
+\Big(-\xi\bar{v}\phi(\bar{w})(\frac{k\pi}{L})^2-\beta_2\bar{v} \Big)   \int_{0}^{L}\gamma_1\cos{\frac{k\pi x}{L}}dx=0,
\end{equation}
and
{\small\begin{equation}\label{423}
\beta_{31}\bar{w}\!\!\int_{0}^{L}\!\!\varphi_1\cos{\frac{k\pi x}{L}}dx+\beta_{32}\bar{w}\!\!\int_{0}^{L}\!\!\psi_1\cos{\frac{k\pi x}{L}}dx +\Big(d_3(\frac{k\pi}{L})^2+{\alpha_3}\bar{w}\Big)\!\!\int_{0}^{L}\!\!\gamma_1\cos{\frac{k\pi x}{L}}dx=0.
\end{equation}}Since $(\varphi_1,\psi_1,\gamma_1)\in\mathcal{Z}$, we infer from (\ref{414}) that
\begin{align}\label{424}
P_k\int_{0}^{L}\varphi_1\cos{\frac{k\pi x}{L}}dx+Q_k\int_{0}^{L}\psi_1\cos{\frac{k\pi x}{L}}dx+\int_{0}^{L}\gamma_1\cos{\frac{k\pi x}{L}}dx=0,
\end{align}
where $P_k, Q_k$ are given by (\ref{49}) and (\ref{410}).  Combining (\ref{422})--(\ref{424}) gives
\[
\begin{pmatrix}
0 & d_2(\frac{k\pi}{L})^2+{\alpha_2}\bar{v} & -\xi\bar{v}\phi(\bar{w})(\frac{k\pi}{L})^2-\beta_2\bar{v}\\
\beta_{31}\bar{w} & \beta_{32}\bar{w} & d_3(\frac{k\pi}{L})^2+{\alpha_3}\bar{w}\\
P_k & Q_k & 1
\end{pmatrix}
\begin{pmatrix}
\int_{0}^{L}\varphi_1\cos{\frac{k\pi x}{L}}dx\\
\int_{0}^{L}\psi_1\cos{\frac{k\pi x}{L}}dx\\
\int_{0}^{L}\gamma_1\cos{\frac{k\pi x}{L}}dx
\end{pmatrix}
=
\begin{pmatrix}
0\\0\\0
\end{pmatrix}.
\]
The determinant of the coefficient matrix $\mathcal{M}$ to the system above is
\begin{align*}
\text{det}(\mathcal{M})=&P_k\Big(d_2 (\frac{k\pi}{L})^2+{\alpha_2}\bar{v}\Big)\Big( d_3(\frac{k\pi}{L})^2+{\alpha_3}\bar{w}\Big)+Q_k\beta_{31}\bar{w}\Big(-\xi\bar{v}\phi(\bar{w})(\frac{k\pi}{L})^2\\
&-\beta_2\bar{v} \Big)-\beta_{31}\bar{w}\Big(d_2(\frac{k\pi}{L})^2+{\alpha_2}\bar{v} \Big)-P_k\beta_{32}\bar{w}\Big(-\xi\bar{v}\phi(\bar{w})(\frac{k\pi}{L})^2-\beta_2\bar{v} \Big)\\
=& P_k(H_2H_3+\beta_{32}\bar{w}Q_k H_2)-\beta_{31}\bar{w}(Q^2_{k}H_2+H_2)\\
=& \Big(-\frac{\beta_{32}}{\beta_{31}}Q_k-\frac{H_3}{\beta_{31}\bar{w}}\Big)(H_2H_3+\beta_{32}\bar{w}Q_k H_2)-\beta_{31}\bar{w}(Q^2_kH_2+H_2)\\
=&-\Big(\frac{\beta_{32}\bar{w}H_2}{\beta_{31}}+\beta_{31}\bar{w}H_2\Big)Q^2_k-\Big(\frac{H_2H_3^2}{\beta_{31}\bar{w}}+\beta_{31}\bar{w}H_2\Big)-\frac{2\beta_{32}H_2H_3}{\beta_{31}}Q_k<0.
\end{align*}
Therefore we have
\begin{align}\label{425}
\int_{0}^{L}\varphi_1\cos{\frac{k\pi x}{L}}dx=\int_{0}^{L}\psi_1\cos{\frac{k\pi x}{L}}dx=\int_{0}^{L}\gamma_1\cos{\frac{k\pi x}{L}}dx=0,
\end{align}
and this implies that $\mathcal{K}_1=0$ in (\ref{421}).  Thus the bifurcation branch $\Gamma_k(s)$ is of pitch--fork, i.e., being one sided.  Now we present another main result of this paper which states that the stability of the bifurcating solutions depends on the sign of $\mathcal{K}_2$.
\begin{theorem}\label{theorem43}
Suppose that all the conditions in Theorem \ref{theorem42} are satisfied and let $\Gamma_k(s)=\{(u_k(s,k),v_k(s,k), w_k(s,k), \chi_k(s))\}$ be the bifurcation branch given by (\ref{412})--(\ref{413}).  Denote $\chi_0=\min_{k\in \mathbb{N}^+}\{\chi^S_k,\chi^H_k\}$ as in (\ref{33}).  Then it holds that: (i) If $\chi_0=\chi^S_{k_0}<\min_{k\in\mathbb N^+}\chi^H_k$, then $\Gamma_{k_0}(s)$ around $(\bar u,\bar v,\bar w,\chi^S_{k_0})$ is asymptotically stable when $\mathcal{K}_2>0$ and it is unstable when $\mathcal{K}_2<0$, while $\Gamma_k(s)$ around $(\bar u,\bar v,\bar w,\chi^S_k)$ is always unstable for each $k\neq k_0$; (ii) If $\chi_0=\chi^H_{k_1}<\min_{k\in\mathbb N^+} \chi^S_k$, then $\Gamma_k(s)$ around $(\bar u,\bar v,\bar w,\chi^S_k)$ is always unstable for each $k\in\mathbb N^+$.
\end{theorem}
The bifurcation curves $\Gamma_k(s)$ in case \emph{(i)} are illustrated in Figure \ref{fig1} schematically.  Our results suggest that $(\bar u,\bar v,\bar w,\chi^S_k)$ loses its stability to stable steady state bifurcating solution with wave mode number $k_0$ for which $\chi^S_k$ achieves its minimum over $\mathbb N^+$.  When case \emph{(ii)} occurs, we surmise that the stability of the homogeneous solution is lost to stable Hopf bifurcating solutions.  This is rigorously verified in Section \ref{section5}.  We would like to mention that, $\mathcal K_2$ can be evaluated in terms of system parameters and we give the detailed calculations in Appendix.
\begin{figure}[h!]
        \centering
        \begin{subfigure}[b]{0.45\textwidth}
                \includegraphics[width=\textwidth]{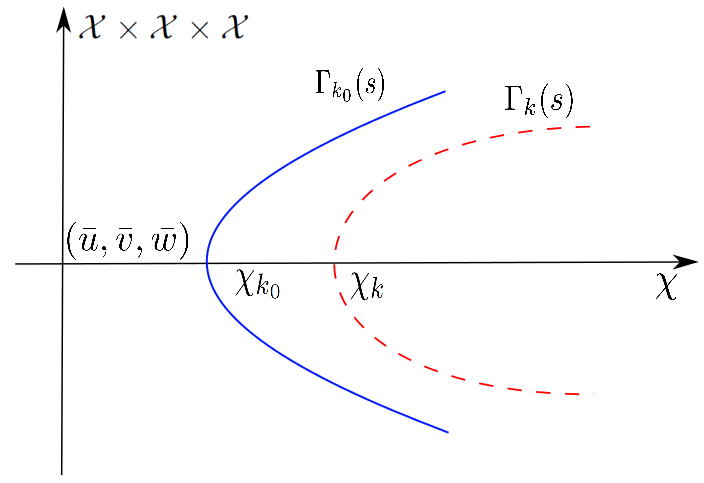}
                \caption*{}
                \label{fig:gull}
        \end{subfigure}\hspace{0.15in}
        \begin{subfigure}[b]{0.45\textwidth}
        \includegraphics[width=\textwidth]{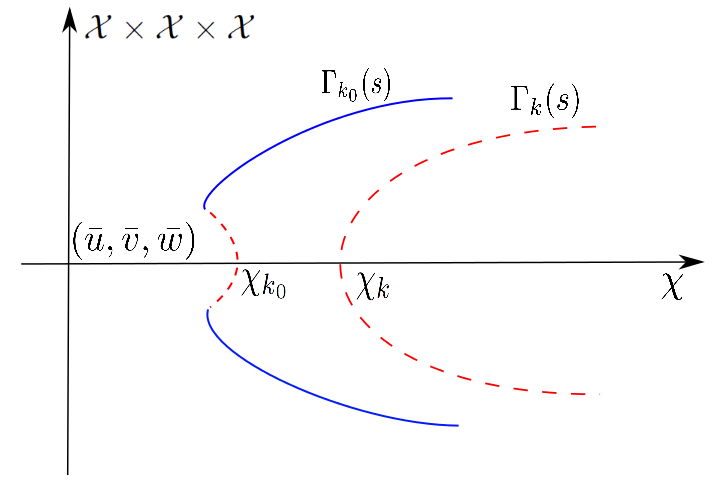}
                \caption*{}
                \label{fig:tiger}
        \end{subfigure}
\vspace*{-10pt} \caption{Pitch--fork bifurcation diagrams when case \emph{(i)} in Theorem \ref{theorem43} occurs.  The stable bifurcation curve is plotted in solid line and the unstable bifurcation curve is plotted in dashed line.  The branch $\Gamma_{k_0}(s)$ around $(\bar u,\bar v,\bar w,\chi^S_{k_0})$ is stable if it turns to the right and is unstable if it turns to the left, while $\Gamma_{k}(s)$ around $(\bar u,\bar v,\bar w,\chi^S_{k})$ is always unstable if $k\neq k_0$.}\label{fig1}
\end{figure}
\begin{proof} [Proof\nopunct] \emph{of Theorem} \ref{theorem43}.
Our proof follows the approaches in \cite{WGY,WZYH} based on slight modifications in the arguments for Corollary 1.13 of \cite{CR2}, or Theorem 3.2 of \cite{WGY}, Theorem 5.5, Theorem 5.6 of \cite{CKWW}.  We shall only prove case \emph{(ii)} and case \emph{(i)} can be treated similarly.

For each $k\in \mathbb N^+$, we linearize (\ref{41}) around $(u_k(s,x),v_k(s,x),w_k(s,x),\chi_k(s))$ and obtain the following eigenvalue problem
\[D\mathcal{F}(u_k(s,x),v_k(s,x),w_k(s,x),\chi_k(s))(u,v,w)=\sigma(s)(u,v,w),~(u,v,w)\in \mathcal{X} \times \mathcal{X}\times \mathcal{X},\]
then $(u_k(s,x),v_k(s,x),w_k(s,x),\chi_k(s))$ is asymptotically stable if and only if the real part of eigenvalue ${\sigma }(s)$ is negative.

Sending $s\rightarrow 0$, we know from the proof of (\ref{415}) that $\bar \sigma=0$ is a simple eigenvalue of $D\mathcal{F}(\bar{u},\bar{v},\bar{w},\chi^S_k)=\sigma(u,v,w)$ or equivalently
\begin{align*}
\begin{split}
\begin{cases}
d_1u''-\chi^S_k \bar{u}\phi(\bar{w})w''-\alpha_1 \bar{u} u+\beta_1\bar{u}w=\sigma u, &x\in(0,L),\\
d_2v''-\xi \bar{v}\phi(\bar{w})w''-\alpha_2 \bar{v} v+\beta_2\bar{v}w=\sigma v, &x\in(0,L),\\
d_3 w''-\beta_{31}\bar{w} u-\beta_{32}\bar{w} v-\alpha_3 \bar{w} w=\sigma w, &x\in(0,L).\\
u'(x)=v'(x)=w'(x)=0,&x=0,L,\\
\end{cases}
\end{split}
\end{align*}
which has one--dimensional eigen--space $\mathcal{N}\big(D\mathcal{F}(\bar{u},\bar{v},\bar{w},\chi^S_k)\big)=\{(P_k,Q_k,1)\cos \frac{k\pi x}{L}\}$.  Multiplying the system above by $\cos \frac{k\pi x}{L}$ and integrating them over $(0,L)$ by parts, we have that ${\sigma }=0$ is an eigenvalue of (\ref{32}) with $\chi=\chi^S_k$ which reads
\[
\begin{pmatrix}
-d_1 (\frac{k\pi}{L})^2-{\alpha_1}\bar{u} & 0 & \chi^S_k \bar{u} \phi(\bar{w}) (\frac{k\pi}{L})^2+\beta_1 \bar{u}\\
0 & -d_2 (\frac{k\pi}{L})^2-{\alpha_2}\bar{v} & \xi \bar{v} \phi(\bar{w}) (\frac{k\pi}{L})^2+\beta_2 \bar{v}\\
-\beta_{31}\bar{w} & -\beta_{32}\bar{w} & -d_3 (\frac{k\pi}{L})^2-{\alpha_3}\bar{w}
\end{pmatrix}.
\]
If $\chi_0=\min_{k\in\mathbb N^+}\chi^H_k<\chi^S_k$ for all $k\in\mathbb N^+$, or $\chi_0=\min_{k\in\mathbb N^+}\chi^S_k< \chi^H_k$ for $k \neq k_0$, we have from the proof of Proposition \ref{proposition31} that this matrix always has an eigenvalue ${\sigma }$ with positive real part.  From the standard eigenvalue perturbation theory in \cite{Ka}, for $s$ being small, there exists an eigenvalue ${\sigma }(s)$ to the linearized problem above that has a positive real part and therefore $(u_k(s,x),v_k(s,x),w_k(s,x),\chi_k(s))$ is unstable for $s\in(-\delta,\delta)$.
\end{proof}

According to Theorem \ref{theorem43}, the only stable bifurcation branch must be $\Gamma^S_{k_0}(s)$ if $\chi_{k_0}=\min_{k\in\mathbb N^+}\{\chi^S_k,\chi^H_k\}$,  therefore $(\bar u,\bar v,\bar w)$ loses its stability only to nonconstant steady state with wave mode $\cos \frac{k_0\pi x}{L}$.  This gives a wave mode selection mechanism for system (\ref{11}) when $\chi$ is around the bifurcation value.  In general it is very difficult to determine whether $\chi_0$ is achieved at $\chi^S_k$ or $\chi^H_k$.  According to the discussions after Remark \ref{remark31}, if the interval length $L$ is sufficiently small, $\chi_0=\chi^S_1<\min_{k\in\mathbb N^+}\chi^H_k$ and the only stable bifurcating solution has wave mode $\cos \frac{\pi x}{L}$ which is spatially monotone.  The wave mode section mechanism given in Theorem \ref{theorem43} is verified and illustrated in our numerical studies of (\ref{11}) in Section \ref{section6}.

\section{Time--periodic positive solutions}\label{section5}

In this section, we study the periodic orbits of (\ref{31}) that bifurcate from $(\bar u,\bar v,\bar w)$ at $\chi=\chi^H_k$.  We want to show that under proper assumptions on system parameters, the constant equilibrium $(\bar{u}, \bar{v}, \bar{w})$ loses its stability through Hopf bifurcation as $\chi$ comes across
$\chi_0=\min_{k\in \mathbb{N}^+}\{\chi^S_k, \chi^H_k\}$.  To apply the bifurcation theory for (\ref{31}) at point $\chi=\chi^H_k$, we need to verify that the real part of eigenvalue crosses the imaginary axis at $\chi^H_k$.

According to the discussions in Section \ref{section3}, Hopf bifurcation occurs for (\ref{31}) at $(\bar{u}, \bar{v}, \bar{w})$ only if $\chi=\chi_{k}^{H}$ and $\eta_1(\chi,k)>0$, when the stability matrix (\ref{32}) has purely imaginary eigenvalues given by
\[\sigma_{1}^{H}(\chi^H_k,k)=-\eta_2(\chi^H_k,k)<0,\sigma_{2,3}^{H}(\chi^H_k,k)=\pm\sqrt{\eta_1(\chi^H_k,k)}i.\]
To determine when $\eta_1(\chi^H_k,k)>0$, we let $\chi^M_k$ be the unique root of $\eta_1(\chi,k)=0$ which is given explicitly in the following form
\begin{align*}
 \chi^M_k=&-\frac{(d_{1}(\frac{k\pi}{L})^2+\alpha_1\bar{u})(d_{2}(\frac{k\pi}{L})^2+\alpha_2\bar{v})+(d_{1}(\frac{k\pi}{L})^2+
 \alpha_1\bar{u})(d_{3}(\frac{k\pi}{L})^2+\alpha_3\bar{w})}{\beta_{31}\bar{u}\bar{w}\phi(\bar{w})(\frac{k\pi}{L})^2}\\
 &-\frac{(d_{2}(\frac{k\pi}{L})^2+\alpha_2\bar{v})(d_{3}(\frac{k\pi}{L})^2+\alpha_3\bar{w})}{\beta_{31}\bar{u}\bar{w}\phi(\bar{w})
 (\frac{k\pi}{L})^2}-\frac{\xi\beta_{32}\bar{v}\phi(\bar{w})(\frac{k\pi}{L})^2+\beta_{32}\beta_{2}\bar{v}+\beta_{31}\beta_{1}\bar{u}}{\beta_{31}\bar{u}\phi(\bar{w})(\frac{k\pi}{L})^2}.
\end{align*}
We first give the following fact which will be used in our coming analysis.
\begin{lemma}\label{lemma51}
Let $\chi^M_k$ be given as above, then for each $k\in \mathbb{N}^+$,it holds that either $\chi_{k}^{H}<\chi^M_k<\chi_{k}^{S}$ or $\chi_{k}^{S}<\chi^M_k<\chi_{k}^{H}$.  Moreover $\eta_1(\chi_{k}^{H},k)>0>\eta_1(\chi_{k}^{S},k)$ if $\chi_{k}^{H}<\chi^M_k<\chi_{k}^{S}$ and $\eta_1(\chi_{k}^{S},k)>0>\eta_1(\chi_{k}^{H},k)$ if $\chi_{k}^{S}<\chi^M_k<\chi_{k}^{H}$.
\end{lemma}
According to Lemma \ref{lemma51} and discussions in Section \ref{section3}, the stability matrix (\ref{32}) has a pair of purely imaginary eigenvalues if and only if $\chi=\chi^H_k<\chi^S_k$, therefore Hopf bifurcation may occur at $(\bar{u}, \bar{v}, \bar{w}, \chi_{k}^{H})$ only when $\chi_{k}^{H}<\chi_{k}^{S}$.  We shall always assume this condition in the coming Hopf bifurcation analysis.

\subsection{Hopf bifurcation}
In this subsection, we prove the existence of Hopf bifurcation of (\ref{31}) assuming that $\chi_{k}^{H}<\chi_{k}^{S}$.  We recall the notation of Sobolev space $\mathcal X=\{u\in H^2(0, L)|u'(0)=u'(L)=0\}$ from Section \ref{section4}.  According to the proof of Theorem \ref{theorem21}, we know that (\ref{31}) is normally parabolic, therefore we can apply the Hopf bifurcation theory from \cite{A} (or Theorem 1.11 from \cite{CR4}, Theorem 6.1 from \cite{LSW}).  Our main result on the existence of nontrivial periodic orbits of (\ref{31}) states as follows.
\begin{theorem}\label{theorem52}
Suppose that all parameters in (\ref{31}) are positive, $\alpha_3>\beta_{31}+\beta_{32}$ and $\phi(\bar w)<0$.  Assume that $\chi^H_k\neq\chi_j^H$ for $\forall j\neq k$ and $\chi_{k}^{H}<\chi_{k}^{S}$, then there exist a positive constant $\delta$ and a unique one--parameter family of nontrivial periodic orbits $\vartheta_k(s)=(\textbf{u}_k(s,x,t), T_k(s), \chi_k(s)), s\in(-\delta, \delta)\rightarrow C^3(\mathbb{R}, \mathcal X^3)\times\mathbb{R}^+\times\mathbb{R}$ with
\begin{equation}\label{51}
\textbf{u}_k(s,x,t)=(\bar{u},\bar{v},\bar{w})+s(V_k^+e^{i\tau_0t}+V_k^-e^{-i\tau_0t})\cos\frac{k\pi x}{L}+o(s)
\end{equation}
such that $(\textbf{u}_k(s,x,t), \chi_k(s))$ is a nontrivial solution of (\ref{31}) and $\textbf{u}_k(s,x,t)$ is periodic of time $t$ with period
\begin{equation}\label{52}
  T_k(s)\approx\frac{2\pi}{\tau_0}, \tau_0=\sqrt{\eta_1(\chi_k^{H},k)}
\end{equation}
and $\{(V_k^\pm,\pm i\tau_0)\}$ are eigen--pairs of matrix (\ref{32}); moreover $\vartheta_k(s_1)\neq\vartheta _k(s_2)$ for all $s_1\neq s_2\in (-\delta, \delta)$ and all nontrivial periodic solutions around $(\bar{u}, \bar{v}, \bar{w}, \chi^H_k)$ must be on the orbit $\vartheta_k(s), s\in(-\delta, \delta)$. In other words, if (\ref{31}) has a nontrivial periodic solution $\textbf{u}_1(x,t)$ with period $T$ for some $\chi\in \mathbb{R}$ around $\vartheta_k(s)$ and a small positive constant $\epsilon$ such that $|\chi-\chi^H_k(s)|<\epsilon, |T-\frac{2\pi}{\tau_0}|<\epsilon$ and $\max_{t\in\mathbb{R}^+, x\in\bar{\Omega}}|\textbf{u}_1(x,t)-(\bar{u}, \bar{v}, \bar{w})|<\epsilon$, then there exist constants $s_0\in(-\delta,\delta)$ and $\theta_0\in[0, 2\pi)$ such that $(T, \chi)=(T_k(s_0), \chi^H_k(s_0))$ and $\textbf{u}_1(x,t)=\textbf{u}_k(s_0,x,t+\theta_0).$
\end{theorem}
\begin{proof}
We follow the approach in the proof of Theorem 5.2 in \cite{WYZ} or Theorem 3.4 in \cite{LSW}.  According to Proposition \ref{proposition31} and Remark \ref{remark31}, the stability matrix (\ref{32}) with $\chi=\chi^H_k$ has a pair of purely imaginary eigenvalues $\sigma_{2,3}^{H}(k)=\pm\sqrt{\eta_1(\chi^H_k)}i$; moreover since $\chi^H_k\neq\chi_j^H$ for $\forall j\neq k$, matrix (\ref{32}) has no eigenvalue of the form $m\tau_0i$ for $m\in \mathbb{N}^+ \backslash \{\pm1\}$.

Let $\sigma^H_1(\chi,k)$ and $\sigma^H_{2,3}(\chi,k)=\lambda(\chi,k)\pm i\tau(\chi,k)$ be the unique eigenvalues of (\ref{32}) in a neighbourhood of $\chi=\chi^H_k$.  Then $\sigma^H_1$, $\lambda$ and $\tau$ are real analytical functions of $\chi$ satisfying $\lambda(\chi^H_k,k)=0$ and $\tau(\chi^H_k,k)=\tau_0>0$.  In order to apply Hopf bifurcation theory, we need to prove the following transversality condition
\begin{equation}\label{53}
  \frac{\partial\lambda(\chi,k)}{\partial\chi}\Big\vert_{\chi=\chi^H_k}\neq0.
\end{equation}
Substituting the eigenvalues $\sigma^H_1(\chi,k)$ and $\sigma^H_{2,3}(\chi,k)=\lambda(\chi,k)\pm i\tau(\chi,k)$ into the characteristic equation of the stability matrix (\ref{32}) and equating the real and imaginary parts give
\begin{equation}\label{54}
\left\{
\begin{array}{ll}
  -\eta_2(\chi,k)=2\lambda(\chi,k)+\sigma^H_1(\chi,k), \\
  \eta_1(\chi,k)=\lambda^2(\chi,k)+\tau^2(\chi,k)+2\lambda(\chi,k)\sigma^H_1(\chi,k), \\
  -\eta_0(\chi,k)=(\lambda^2(\chi,k)+\tau^2(\chi,k))\sigma^H_1(\chi,k).
\end{array}
\right.
\end{equation}
Differentiating the equations above with respect to $\chi$, we obtain
\[2\lambda'(\chi,k)+\sigma'_1(\chi,k)=0\]
and
\begin{align}\label{55}
&2\lambda(\chi,k)\lambda'(\chi,k)+2\tau(\chi,k)\tau'(\chi,k)+2\lambda'(\chi,k)\sigma_1(\chi,k)+2\lambda(\chi,k)\sigma_1'(\chi,k) \nonumber\\
=&\beta_{31}\bar{u}\bar{w}\phi(\bar{w})(\frac{k\pi}{L})^2, \nonumber\\
&(2\lambda(\chi,k)\lambda'(\chi,k)+2\tau(\chi,k)\tau'(\chi,k))\sigma_1(\chi,k)+(\lambda^2(\chi,k)+\tau^2(\chi,k))\sigma_1'(\chi,k) \nonumber\\
=&-\beta_{31}\bar{u}\bar{w}\phi(\bar{w})(\frac{k\pi}{L})^2(d_2(\frac{k\pi}{L})^2+{\alpha_2}\bar{v}).
\end{align}
Since $\lambda(\chi^H_k,k)=0$ and $\sigma_1(\chi^H_k,k)=-\eta_2(\chi^H_k,k)$, solving (\ref{55}) with $\chi=\chi^H_k$ gives that
\begin{align}\label{56}
\sigma_1'(\chi^H_k,k)=& -\frac{\beta_{31}\bar{w}\bar{u}\phi(\bar{w})(\frac{k\pi}{L})^2(d_2(\frac{k\pi}{L})^2
+{\alpha_2}\bar{v}-\eta_2(\chi^H_k,k))}{\tau_0^2+\eta^2_2(\chi^H_k,k)}\nonumber\\
=&\frac{\beta_{31}\bar{w}\bar{u}\phi(\bar{w})(\frac{k\pi}{L})^2((d_1+d_3)(\frac{k\pi}{L})^2
+\alpha_1\bar u+\alpha_3\bar w)}{\tau_0^2+\eta^2_2(\chi^H_k,k)}<0
\end{align}
and
\[\lambda'(\chi^H_k,k)=-\frac{1}{2}\sigma_1'(\chi^H_k)>0.\]
This verifies all the transversality conditions required in applying the Hopf bifurcation theory, then Theorem \ref{theorem52} follows from Theorem 1 in \cite{A}.
\end{proof}

Theorem \ref{theorem52} implies that system (\ref{31}) admits time--periodic spatial patterns that bifurcate from $(\bar{u},\bar{v},\bar{w},\chi^H_k)$ if and only if $\chi^H_k<\chi^S_k$. Furthermore, it gives the explicit expression of the time--periodic spatial patterns as $\vartheta_k$ mentioned above with the spatial profile of eigen--function $\cos\frac{k\pi x}{L}$.

As we have discussed in Section \ref{section3}, it is very difficult to determine the necessary condition $\chi^H_k<\chi^S_k$ in terms of system parameters, however if the interval $L$ is sufficiently small, we always have $\chi^H_k>\chi^S_k$ for each $k\in\mathbb N^+$, and therefore this indicates that Hopf bifurcation dose not occur for (\ref{31}) when the interval length is sufficiently small.  Indeed, in this case, we already know from the discussions after the proof of Theorem \ref{theorem43} that the stability of the homogeneous solution $(\bar{u},\bar{v},\bar{w})$ is lost through the steady state bifurcation at the first bifurcation branch $(\bar{u},\bar{v},\bar{w},\chi_{1}^S)$, which contains stable stationary solutions of (\ref{31}) with eigenfunction $\cos(\frac{\pi x}{L})$.

\subsection{Stability of time--periodic bifurcating solutions}
We continue to explore the stability of the time--periodic bifurcating solutions on the bifurcation curves $\vartheta_k(s)$ obtained in Theorem \ref{theorem52}.  The stability here we mean is the formal linearized stability of a periodic solution relative to perturbations from $\vartheta_k(s)$.  Suppose that $\chi_{k_0}^H=\min_{k\in \mathbb{N}^+}\chi^H_k<\chi^S_k, \forall k\in \mathbb{N}^+$, and assume that all the conditions in Theorem \ref{theorem52} are satisfied here, then our stability results show that $\vartheta_k(s)$, $s\in(-\delta, \delta)$ is asymptotically stable only if $\chi=\chi_{k_0}^H$.

Denote $\textbf{u}_k(s,x,t)=(u_k(s,x,t),v_k(s,x,t),w_k(s,x,t))$ and let $(\textbf{u}_k(s,x,t),T_k(s),\break\chi_k(s))$ be the periodic solutions on the branch $\vartheta_k(s)$ obtained in Theorem \ref{theorem52}.  Then we can rewrite (\ref{31}) into the following form
\[\frac{d\textbf{u}_k}{dt}=\mathcal G(\textbf{u}_k,\chi_k(s)),\]
where
\begin{align*}
\mathcal G(\textbf{u}_k,\chi_k(s))=
\begin{pmatrix}
(d_1 u'_k -\chi_k(s) u_k\phi(w_k)w'_k)'+\alpha_1(1-u_k)u_k+\beta_1u_k w_k \\
(d_2 v'_k -\xi v_k\phi(w_k)w'_k)'+\alpha_2(1-v_k)v_k+\beta_2v_k w_k \\
d_3 w''_k+\alpha_3(1-w_k)w_k-\beta_{31}u_k w_k-\beta_{32}v_k w_k \\
\end{pmatrix}.
\end{align*}
Differentiating the system against $t$, writing $\dot{\textbf{u}}_k=\frac{d \textbf{u}}{dt}$, we have
\[\frac{d\dot{\textbf{u}}_k}{dt}=\mathcal G_u(\textbf{u}_k,\chi_k(s))\dot{\textbf{u}}_k,\]
then we observe that 0 is a Floquet exponent and 1 is a Floquet multiplier for $\textbf{u}_k$.

Linearize the periodic solution around the bifurcation branch $\vartheta_k(s)$ by substituting the perturbed solution $\textbf{u}_k+\textbf{w}e^{-lt}$, where $\textbf{w}$ is a sufficiently small $T$-periodic function and $l=l(s)$ is a continuous function of $s$, then we have that
\begin{equation}\label{57}
  \frac{d\textbf{w}(s,t)}{dt}=\mathcal G_u(\textbf{u}_k,\chi_k(s))\textbf{w}(s,t)+l(s)\textbf{w}(s,t),
\end{equation}
where $\mathcal G_u$ is the Fr\'echet derivative with respect to $\textbf{u}$ given by
{\tiny\begin{align*}
&\mathcal G_u(\textbf{u}_k,\chi_k(s))=D\mathcal G(u_k,v_k,w_k,\chi_k(s))(u,v,w)\\
=&\begin{pmatrix}
d_1 u'' -\chi_k(s)(u\phi(w_k)w'_k+u_k w\phi'(w_k)w'_k+u_k\phi(w_k)w')'+(\alpha_1
-2\alpha_1u_k+\beta_1 w_k)u+\beta_1u_kw \\
d_2 v'' -\xi(v\phi(w_k)w'_k+v_k w\phi'(w_k)w'_k+v_k \phi(w_k)w')'+(\alpha_2
-2\alpha_2v_k+\beta_2 w_k)v+\beta_2v_kw \\
d_3 w''-\beta_{31}u_k w-\beta_{32}v_k w+(\alpha_3-2\alpha_3w_k-\beta_{31}u_k-\beta_{32}v_k)w \\
\end{pmatrix}.
\end{align*}}The stability of the bifurcating solutions around $\chi^H_k$ can be determined by computing the eigenvalues of this reduced equation. When $s=0$, (\ref{57}) is associated with the eigenvalue problem
\begin{equation}\label{58}
  \mathcal G_0(k)\textbf{w}=l(0)\textbf{w},
\end{equation}
where
\begin{align*}
\mathcal G_0(k)=
\begin{pmatrix}
d_1\frac{d^2}{dx^2}-{\alpha_1}\bar{u}& 0 &-\chi^H_k\bar{u}\phi(\bar{w})\frac{d^2}{dx^2}+\beta_1\bar{u}\\
0& d_2\frac{d^2}{dx^2}-{\alpha_2}\bar{v}& -\xi\bar{v}\phi(\bar{w})\frac{d^2}{dx^2}+\beta_2\bar{v} \\
-\beta_{31}\bar{w}& -\beta_{32}\bar{w}& -d_3\frac{d^2}{dx^2}-{\alpha_3}\bar{w} \\
\end{pmatrix},
\end{align*}
the spectrum of which is infinitely dimensional.  Moreover $\mathcal G_0$ corresponds to the stability matrix (\ref{32}).
\begin{align}\label{59}
\mathcal{A}_j(\chi^H_k)=
\begin{pmatrix}
-d_1(\frac{j\pi}{L})^2-{\alpha_1}\bar{u}& 0 &\chi^H_k\bar{u}\phi(\bar{w})(\frac{j\pi}{L})^2+\beta_1\bar{u}\\
0& -d_2(\frac{j\pi}{L})^2-{\alpha_2}\bar{v}& \xi\bar{v}\phi(\bar{w})(\frac{j\pi}{L})^2+\beta_2\bar{v} \\
-\beta_{31}\bar{w}& -\beta_{32}\bar{w}& d_3(\frac{j\pi}{L})^2-{\alpha_3}\bar{w} \\
\end{pmatrix}, j\in \mathbb{N}^+.
\end{align}

Suppose that $\min_{k\in \mathbb{N}^+}\{\chi^H_k,\chi^S_k\}=\chi_{k_0}^H$ for some $k_0\in \mathbb{N}^+$.  We first show that $\vartheta_k(s)$ around $\chi^H_k$ is unstable for any $k\neq k_0$.  Denote the eigenvalues of $\mathcal{A}_k(\chi^H_k)$ by $\sigma^H_1(\chi^H_k,k)$, $\sigma^H_2(\chi^H_k,k)$ and $\sigma^H_3(\chi^H_k,k)$.  According to the Proposition \ref{proposition31}, there exists at least one eigenvalue with positive real part if $\chi>\chi_0$.  Therefore for any positive integer $k\neq k_0$, we have that $\mathcal G_0(k)$ must have an eigenvalue with positive real part hence $l(0)<0$ if $k\neq k_0$. By the standard perturbation theory for an eigenvalue of finite multiplicity \cite{Henry,Ka}, $l(s)<0$ for $s$ being small if $k\neq k_0$, therefore all the bifurcation branches $\vartheta_k(s)$ around $(\bar{u},\bar{v},\bar{w})$ are unstable if $k\neq k_0$. The result indicates that if a periodic bifurcation solution is stable, it must be on the $\vartheta_{k_0}(s)$ branch where $\chi_{k_0}^H<\min_{k\in\mathbb{N}^+}\chi^S_k$, i.e., it is on the left-most branch, while the branches on its right hand side are always unstable.

Now we proceed to discuss the stability of branch $\vartheta_{k_0}(s)$ around $(\bar{u},\bar{v},\bar{w},\chi_{k_0}^H)$.  According to Lemma 2.10 in \cite{CR4} (or \cite{JN,JD}), the eigenvalue $l(k)$ is a continuous real function of $s$ near the origin.  For $\chi$ being around $\chi^H_{k_0}$, the eigenvalue of (\ref{59}) are $\sigma_1(\chi,k)$ and $\sigma_{2,3}(\chi,k)=\lambda(\chi,k)\pm i\tau(\chi,k)$.  According to Theorem 2.13 in \cite{CR4}, $l(s)$ and $s\chi'_{k_0}(s)$ have the same zeros in small neighbourhood of $s=0$ where $l(k)$ and $-\lambda(\chi_{k_0}^H)s\chi'_{k_0}s(s)$ have the same sign $-\lambda(\chi_{k_0}^H)s\chi'_{k_0}s(s)\neq 0, l(s)\neq 0$, and \[|l(s)+\lambda((\chi_{k_0}^H)s\chi'_{k_0}s(s)|\leq |s\chi'_{k_0}(s)|o(1),~\text{as}~~ s\rightarrow 0.\]

According to Theorem 8.2.3 in \cite{Henry}, if $l(s)>0$, the periodic bifurcation solutions are orbitally asymptotically stable and, if $l(s)<0$, the periodic bifurcation solutions are orbitally unstable.  We have proved that $\lambda(\chi_{k_0}^H)<0$ and $l(s)$ and $s\chi'_{k_0}(s)$ has the same sign. Therefore, assuming that $\chi''_{k_0}(0)\neq 0$, if the branching solutions appear supercritical, they are stable and if they appear subcritical, they are unstable.  Therefore, one need to compute $\chi'_{k_0}$ and/or $\chi''_{k_0}$ similarly as in Section \ref{section4}.  The calculations are straightforward but complicated and we skip them here for simplicity.

\section{Numerical simulations}\label{section6}
This section is devoted to the numerical studies of system (\ref{31}).  We are motivated to investigate the effects of prey--taxis on the formation of nontrivial patterns to this system.  In particular, we show that the one--dimensional system admits both stationary and time--periodic solutions emerging from bifurcations.  Moreover, we shall see that, when the prey--taxis rate $\chi$ is taken to be greatly larger than the critical bifurcation value $\chi_0$, (\ref{31}) can develop various interesting patterns with striking structures such as spikes, propagation, coarsening, etc.

\subsection{Stationary patterns}
In Table \ref{table1}, we list the values of $\chi^S_k$ and $\chi^H_k$ given in (\ref{34}) and (\ref{35}) respectively.  It is shown that their minimum value over $\mathbb N^+$ is achieved at $\chi^S_{6}\approx6.05$, therefore $(\bar u,\bar v,\bar w)=(1.71,1.18,0.71)$ loses its stability through steady state bifurcation with wave mode $\cos \frac{6\pi x}{7}$.
\begin{table}[h!]
\centering
\begin{tabular}{|c|p{1cm}|p{0.85cm}|p{0.85cm}|p{0.85cm}|p{0.85cm}|p{0.85cm}|p{1cm}|p{1cm}|p{1cm}|} \hline
$k$& 1 & 2 & 3 & 4 & 5& \textbf{6}& 7&8&9\\ \hline
$\chi^S_{k}$&66.98& 18.98 & 10.30&7.49&6.41&\textbf{6.05}&6.09&6.36&6.80 \\ \hline
$\chi^H_{k}$&1204.20& 550.84 &504.20&575.80&705.50&878.48&1089.70&1336.90&1619.19\\
\hline
\end{tabular}
\caption{Values of $\chi^S_k$ in (\ref{34}) and $\chi^H_k$ in (\ref{35}) for $L=7$.  The system parameters are chosen to be $d_1=d_3=0.1$, $d_2=2$, $\alpha_1=\beta_1=\beta_2=0.5$, $\alpha_2=2$, $\alpha_3=1$ and $\beta_{31}=\beta_{32}=0.1$, $\xi=0.5$.  The sensitivity function is $\phi(w)=w(0.1-w)$ which models the group defense of the preys when population density surpasses 0.1.  We see that $\min_{k\in\mathbb N^+}\{\chi^S_k,\chi^H_k\}=\chi^S_6$.  Therefore $(\bar u,\bar v,\bar w)$ loses its stability to the stable wave mode $\cos \frac{6\pi x}{7}$.  This is numerically verified in Figure \ref{figure2}.}
\label{table1}\vspace*{-10pt}
\end{table}
In Figure \ref{figure2}, we plot numerical solutions of (\ref{31}) subject to initial data $(u_0,v_0,w_0)=(\bar u, \bar v,\bar w)+(0.01,0.01,0.01)\cos \pi x$, which are small perturbations from the homogeneous equilibrium.  We see that the initial data have spatial profiles in the form of $\cos \pi x$, but the spatial--temporal patterns develop according to the stable wave mode $\cos \frac{6\pi x}{7}$.
\begin{figure}[h!]
\centering
\includegraphics[width=\textwidth,height=2.5in]{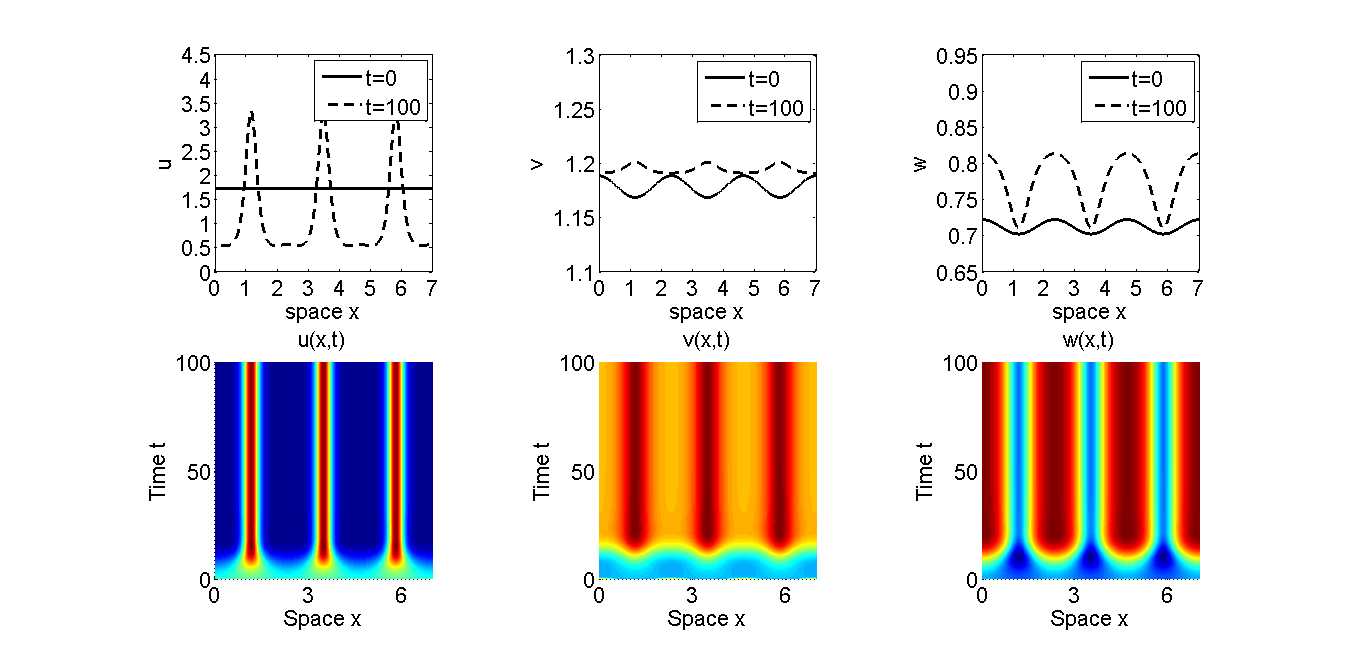}
  \caption{Formation of stationary patterns of (\ref{31}) over $\Omega=(0,7)$.  System parameters in all the graphes are taken to be the same as in Table \ref{table1} except that $\chi=8$, which is larger than $\chi^S_{k_0}\approx6.05$ given in Table \ref{table1}.  Initial data are $(u_0,v_0,w_0)=(\bar u, \bar v,\bar w)+(0.01,0.01,0.01)\cos \pi x$, while the stable pattern has wave mode $\cos \frac{6\pi x}{7}$.  These graphes support our stability analysis of the bifurcating solutions.}\label{figure2}
\end{figure}

\begin{table}[h!]
\centering
\begin{tabular}{|c|c|c|c|c|c|c|c|c|}
  \hline
Interval length $L$& 1 & 2 & 3& 4&5&6&7&8\\ \hline
    $k_0$& 1 & 2 & 3& 4&5&5&6&7\\ \hline
    $\chi_0=\chi^S_{k_0}$&6.09&6.09&6.09&6.09&6.09&6.08&6.05&6.04\\ \hline
Interval length $L$&9&10&11&12&13&14&15&16\\ \hline
    $k_0$&8&9&10&11&12&13&14&15\\ \hline
    $\chi_0=\chi^S_{k_0}$&6.04&6.03&6.04&6.04&6.04&6.04&6.04&6.04\\ \hline
\end{tabular}
\caption{Stable wave mode numbers and the corresponding bifurcation values $\chi_0$ for different interval lengthes.  System parameters are chosen to be the same as those in Table \ref{table1}.  We see that the threshold value $\chi_0$ is always achieved at the steady state bifurcation point $\chi^S_{k_0}$.  This table also indicates that larger intervals support higher wave modes.}
\label{table2}
\end{table}
Numerical simulations in Figure \ref{figure3} are devoted to verifying that $(\bar u,\bar v,\bar w)$ loses its stability to steady state bifurcations when $\chi_0$ is achieved at $\chi^S_{k_0}$ for this set of system parameters.  These simulations support the results on the stability of the bifurcating solutions in Theorem \ref{theorem43}.
\begin{figure}[h!]
\centering
\includegraphics[width=\textwidth,height=2.5in]{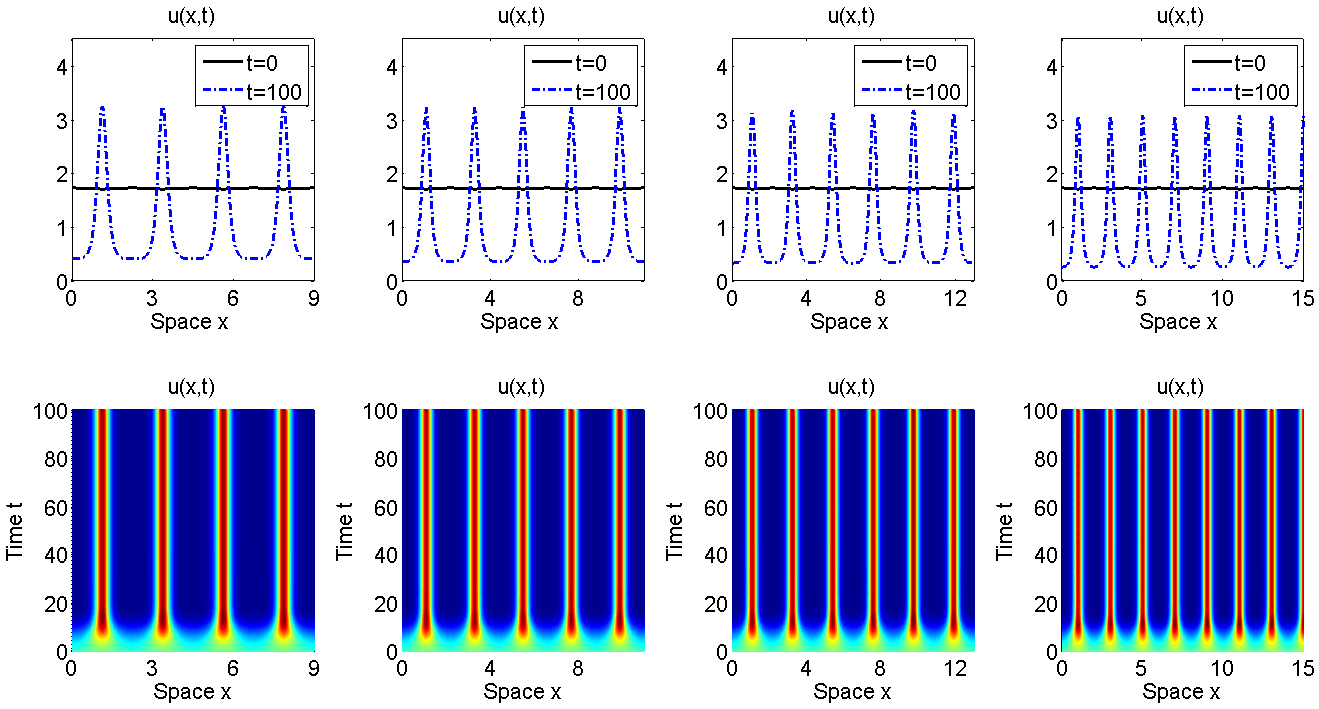}
  \caption{Formation of stationary patterns of (\ref{31}) over intervals with lengthes $L=9$, $11$, $13$ and $15$.  System parameters here are taken to be the same as those in Table \ref{table1} except that $\chi=8$, which is slightly larger than $\chi^S_{k_0}\approx6.04$ given in Table \ref{table2}.  Initial data are $(u_0,v_0,w_0)=(\bar u, \bar v,\bar w)+(0.01,0.01,0.01)\cos \pi x$.  These graphes support our stability analysis of the bifurcating solutions and indicate that large intervals support more aggregates than small intervals.}\label{figure3}
\end{figure}

\subsection{Time periodic patterns}
Our next set of numerical results are provided to demonstrate that (\ref{31}) admits time--periodic patterns through Hopf bifurcations.  To this end, we set system parameters to be $d_1=d_3=1$, $d_2=0.1$, $\alpha_1=0.02$, $\alpha_2=0.04$, $\alpha_3=8$ and $\beta_1=0.05$, $\beta_2=\beta_{31}=\beta_{32}=0.5$, while the sensitivity function is chosen $\phi(w)=0.05w(0.2-w)$.  We shall show that the equilibrium $(\bar u,\bar v,\bar w)=(2.13,6.65,0.45)$ loses its stability to time--periodic orbits.  Table \ref{table3} lists the values of $\chi^S_k$ and $\chi^H_k$ when the interval length is $L=7$.  We see that the threshold value $\chi_0$ is achieved at $\chi^H_3\approx 92.57$.
\begin{table}[h!]
\centering
\begin{tabular}{|c|c|c|c|c|c|c|c|c|c|} \hline
$k$& 1 & 2 & \textbf{3} & 4 & 5& 6& 7&8&9\\ \hline
$\chi^S_{k}$&106.4&98.63&107.32&122.53&143.15&169.05&200.24&236.80&278.76 \\ \hline
$\chi^H_{k}$&186.37&96.73&\textbf{92.57}&105.46&127.03&155.24&189.40&229.24&274.64\\
\hline
\end{tabular}
\caption{Values of $\chi^S_k$ in (\ref{34}) and $\chi^H_k$ in (\ref{35}) for $L=7$.  System parameters are $d_1=d_3=1$, $d_2=0.01$, $\alpha_1=0.02$, $\alpha_2=0.04$, $\alpha_3=8$ and $\beta_1=0.05$, $\beta_2=\beta_{31}=\beta_{32}=0.5$, while the sensitivity function is $\phi(w)=w(0.2-w)$.  We see that $\min_{k\in\mathbb N^+}\{\chi^S_k,\chi^H_k\}=\chi^H_3$.  Therefore $(\bar u,\bar v,\bar w)$ loses its stability to the time--periodic solutions with wave mode $\cos \frac{3\pi x}{7}$.  This is numerically verified in Figure \ref{figure4}.}
\label{table3}\vspace*{-10pt}
\end{table}
In Figure \ref{figure4}, we plot the numerical solutions of (\ref{31}) subject to initial data $(u_0,v_0,w_0)=(\bar u, \bar v,\bar w)+(0.01,0.01,0.01)\cos \pi x$, small perturbations from the homogeneous equilibrium.  The initial data have spatial profiles in the form of $\cos \pi x$, but the spatial--temporal patterns develop according to the stable time--periodic patterns with wave mode $\cos \frac{3\pi x}{7}$.
\begin{figure}[h!]
\centering
\includegraphics[width=\textwidth,height=2.5in]{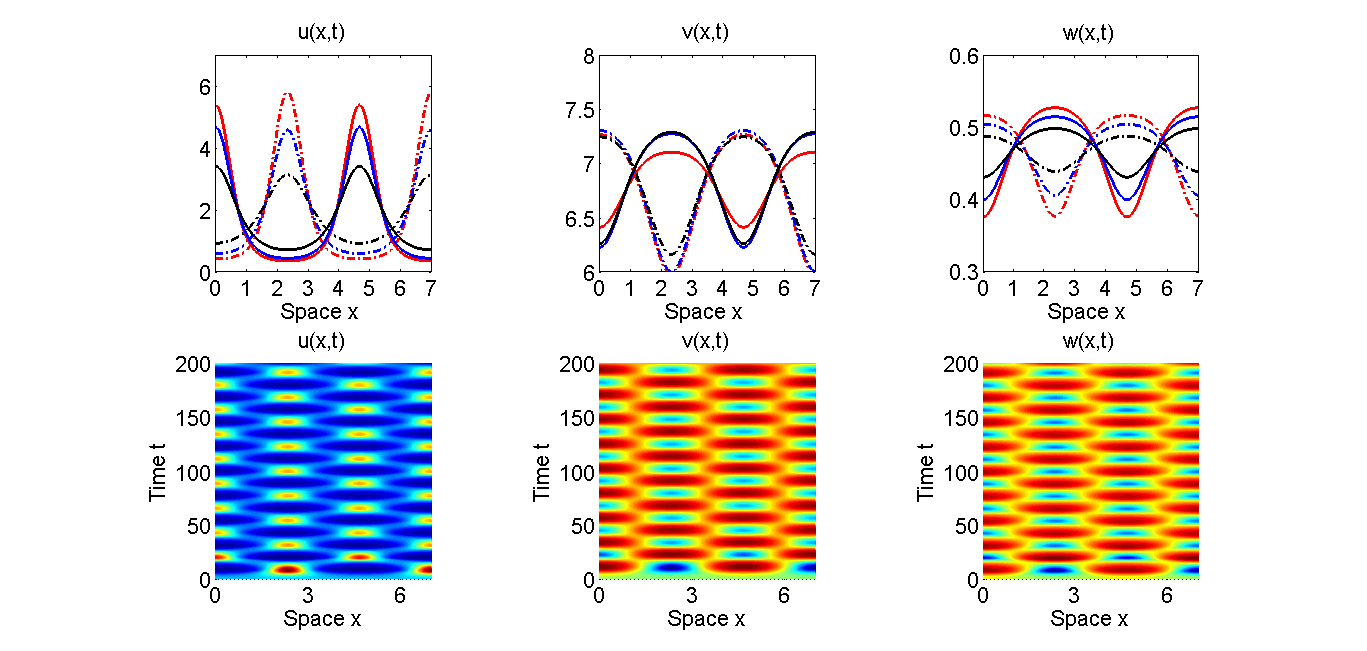}
  \caption{Formation of time--periodic spatial patterns of (\ref{31}) over $\Omega=(0,7)$. System parameters in all the graphes are taken to be the same as in Table \ref{table3} except that $\chi=120$, which is slightly larger than $\chi^H_{k_0}\approx 92.57$ given in Table \ref{table3}.   Initial data are $(u_0,v_0,w_0)=(\bar u, \bar v,\bar w)+(0.01,0.01,0.01)\cos \pi x$, however the stable oscillating patterns have spatial profile $\cos\frac{3\pi x}{7}$, which emerge periodically.  These plots support our stability analysis in Section \ref{section5}.}\label{figure4}
\end{figure}

\begin{table}[h!]
\centering
\begin{tabular}{|c|c|c|c|c|c|c|c|c|}
  \hline
Interval length $L$& 2 & 3& 4&5&6&7&8&9\\ \hline
    $k_0$& 1 & 2 & 3& 4&5&5&6&7\\ \hline
    $\chi_0=\chi^H_{k_0}$&97.68&92.13&97.68&91.49&92.13&92.57&91.15&92.13\\ \hline
Interval length $L$&10&11&12&13&14&15&16&17\\ \hline
    $k_0$&8&9&10&11&12&13&14&15\\ \hline
    $\chi_0=\chi^H_{k_0}$&91.5&91.2&91.13&91.21&91.30&91.49&91.15&91.40\\ \hline
\end{tabular}
\caption{Stable wave mode numbers and the corresponding bifurcation values $\chi_0$ for different interval lengthes, where system parameters are chosen to be the same as those in Table \ref{table3}.  We see that the threshold value $\chi_0$ is always achieved at the Hopf bifurcation point $\chi^H_{k_0}$.}
\label{table4}
\end{table}
Numerical simulations in Figure \ref{figure5} are devoted to verifying that the $(\bar u,\bar v,\bar w)$ loses its stability to Hopf bifurcations when $\chi_0$ is achieved at $\chi^H_{k_0}$, where system parameters are taken to be the same as those in Table \ref{table3}.  In particular, we select the interval lengths to be $L=9$, 11, 13 and 15 respectively.  These simulations support our results on the stability of the Hopf bifurcating solutions obtained in Theorem \ref{theorem52}.
\begin{figure}[h!]
\centering
\includegraphics[width=\textwidth,height=2.5in]{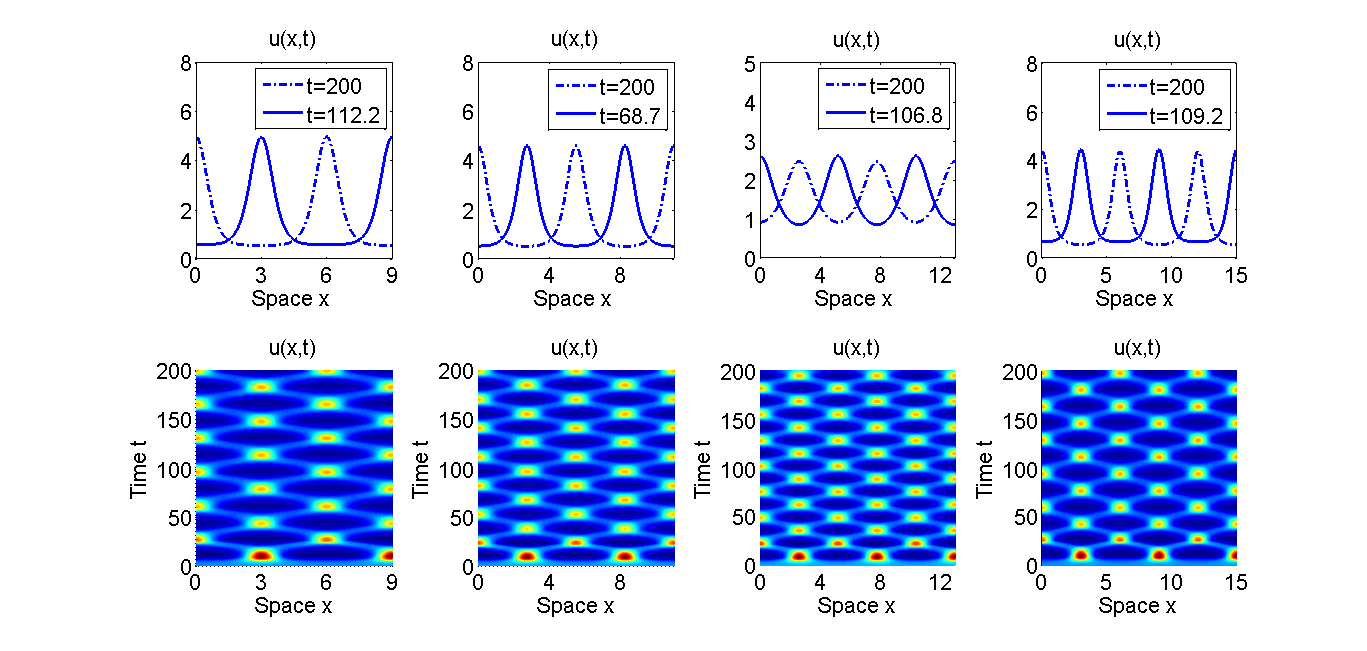}
  \caption{Formation of time--periodic spatial patterns of (\ref{31}) over intervals with lengthes $L=9$, $11$, $13$ and $15$ respectively.  System parameters in all the graphes are taken to be the same as in Table \ref{table3} except that $\chi=120$, which is slightly larger than $\chi^H_{k_0}$ given in Table \ref{table4}.   Initial data are $(u_0,v_0,w_0)=(\bar u, \bar v,\bar w)+(0.01,0.01,0.01)\cos \pi x$.  These graphes support our stability analysis of the bifurcating solutions.}\label{figure5}
\end{figure}

\subsection{Other interesting patterns}
In Figure \ref{figure6}, we plot the formation of stable boundary spikes of (\ref{31}) through traveling wave over $\Omega=(0,7)$.  System parameters are taken to be $d_1=5$, $d_2 = 0.5$, $d_3 = 1$, $\alpha_1=\alpha_2=0.05$, $\alpha_3 = 4$, $\beta_1 = 0.3$, $\beta_2 = 0.5$, $\beta_{31}=1, \beta_{32}=\frac{1}{3}$, $\xi=0.1$ and $\phi(w)=w(0.1-w)$.  Prey--taxis rate $\chi=3000$ is greatly larger than the critical bifurcation value $\chi_0=956.79$.  We observe that the boundary spike is developed through traveling wave solution.  However, rigorous analysis of qualitative properties of the propagating solutions is out of the scope of our paper.
\begin{figure}[h!]
\centering
\includegraphics[width=\textwidth,height=2.5in]{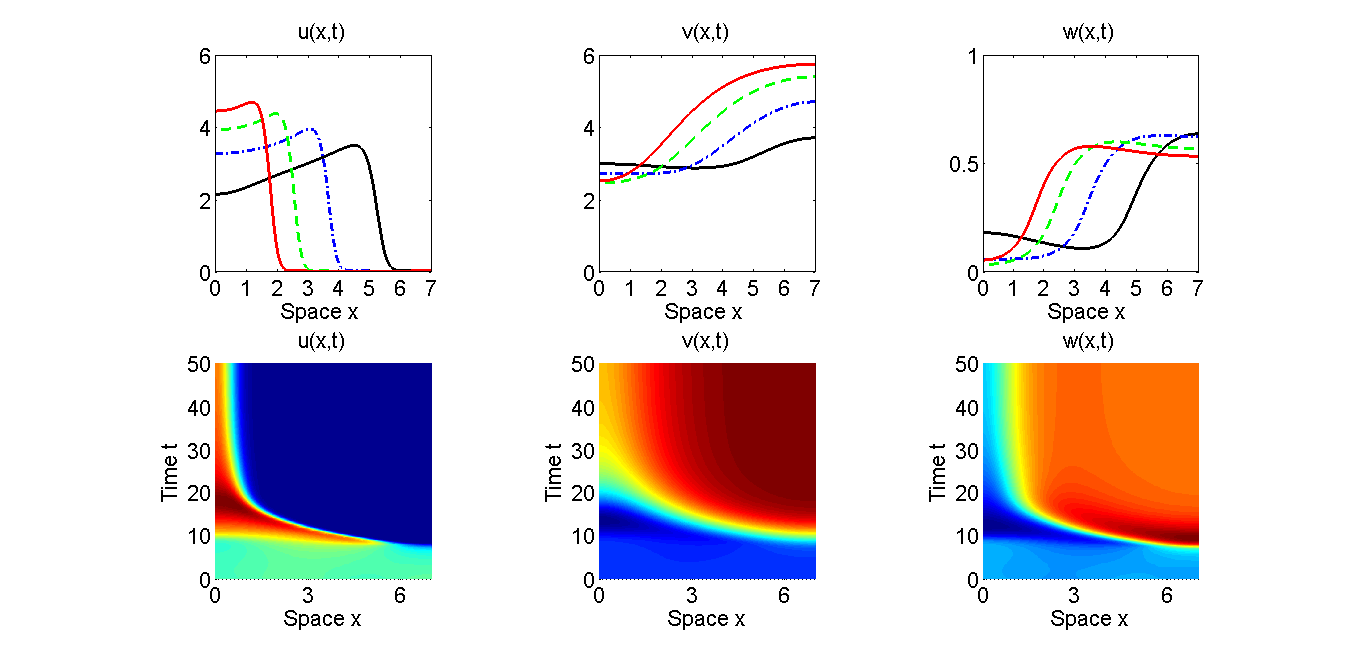}
  \caption{Formation and development of boundary spike through wave propagation.  Initial data are $(u_0,v_0,w_0)=(\bar u, \bar v,\bar w)+(0.01,0.01,0.01)\cos \pi x$.  Prey--taxis rate $\chi=3000$ which is far away from the critical bifurcation value $\chi_0=956.79$.}\label{figure6}
\end{figure}
Finally, we present numerical simulations in Figure \ref{figure7} to show that when the prey--taxis is much larger than $\chi_0$, (\ref{31}) admits some other interesting and striking dynamics such as merging and emerging of spikes, irregular spatial--temporal oscillations etc.  For example, Subplot \emph{(i)} of Figure \ref{figure7} shows that there occurs a coarsening process in (\ref{31}) in which interior spikes of $u(x,t)$ shift to the boundary or the center to merge into another stable spike.  We also observe the spontaneous emergence of stable interior spikes at time $t\approx 100$.  All the parameters and the initial data in (\ref{31}) are chosen to be the same as in Figure \ref{figure2}, except that $\chi=124$, which is much far away from $\chi_0\approx6.05$.  In subplot \emph{(ii)}, when the system parameters and the initial data are chosen to be the same as in Figure \ref{figure4}, except that $\chi=160$, (\ref{31}) initially develops time--periodic spatial patterns which are metastable.  Then the oscillating patterns develop into stable stationary spikes.  Time--periodic patterns and spontaneous initiation of interior multiple spikes are observed in subplots \emph{(iii)} and \emph{(iv)}.
\begin{figure}[h!]
\centering
\includegraphics[width=\textwidth,height=1.8in]{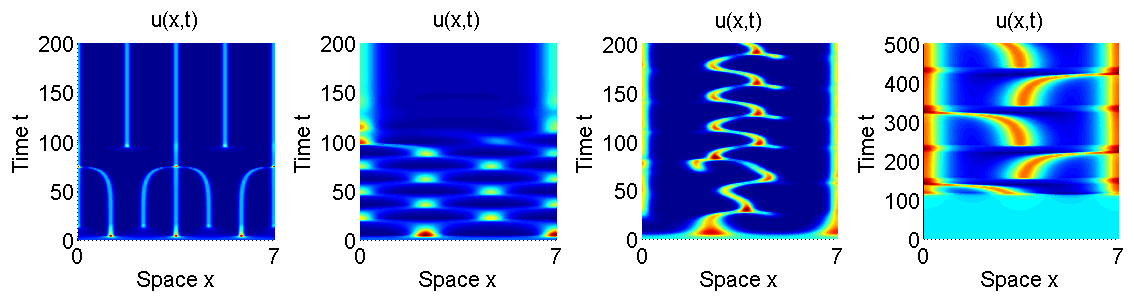}
  \caption{Pattern formations in (\ref{31}) due to the effect of large prey--taxis rate $\chi$.  Various interesting and complex spatial--temporal dynamics are observed in this system.}\label{figure7}
\end{figure}

\section{Conclusions and discussions}\label{section7}

Our paper investigates population dynamics of a two--predator and one--prey model with prey--taxis, given by a $3\times3$ reaction--advection--diffusion system.  It is proved that the system admits positive classical solution which is global and uniformly bounded in time over 1D or 2D bounded domains.  The same results are obtained for its parabolic--parabolic--elliptic counterpart for domains of arbitrary space dimension.

Stability of the unique positive equilibrium is studied when the domain is a finite interval.  It shows that both prey--taxis $\chi$ and sensitivity function $\phi$ determines the linearized stability of this equilibrium.  It is known (see \cite{LHL2} e.g.) that, in contrast to chemotaxis \cite{HP,Ho} or advection for competition system \cite{WGY},  prey--taxis stabilizes constant equilibrium for one--predator and one--prey system.  However, our result reveals that this is true only when there is no group defense in the preys, i.e., a huge amount of preys can aggregate and keep their predators away from the habitat.  If the predators retreat from the habitat, which can be modeled by choosing $\phi(\bar w)<0$, prey--taxis destabilizes the constant equilibrium, which becomes unstable as $\chi$ surpasses $\chi_0=\min_{k\in\mathbb N}\{\chi^S_k,\chi^H_k\}$ given by (\ref{33}).  Therefore group defense is an important mechanism in the formation of nontrivial patterns in (\ref{11}).

We have obtained both stationary and time--periodic spatial patterns to the system over 1D bounded interval through steady state bifurcation at $\chi=\chi^S_k$ and Hopf bifurcation at $\chi=\chi^H_k$ respectively.  Stabilities of these bifurcating solutions are also investigated rigorously.  It is proved that steady state bifurcation occurs at $(\bar u,\bar v,\bar w,\chi^S_k)$ for each $k\in\mathbb N^+$, however, only the $k_0$--branch that turns to the right is stable if $\chi_0=\chi^S_{k_0}<\min_{k\in\mathbb N}\chi^H_k$.  In other words, if the steady state bifurcation curve is stable, it must be on the left most branch on the $\chi$--axis.  On the other hand, Hopf bifurcation $(\bar u,\bar v,\bar w,\chi^H_k)$ only if $\chi^H_k<\chi^S_k$, while only the left most branch can be stable.  Moreover, our analysis indicates that small intervals only supports steady state bifurcations while large intervals may lead to Hopf bifurcation when the system parameters are chosen properly.  Extensive numerical simulations are performed to illustrate and support our theoretical findings.  Apparently, the formation of these nontrivial patterns is due to the effect of large prey--taxis and prey group defense effect.

Global existence and bounded are obtained for (\ref{11}) over 2D and it is interesting to ask the same question for the system over higher dimensions.  Logistic decays in the kinetics help to prevent finite or infinite time blow--ups, however, whether or not they are sufficient over higher dimensions, in particular when the prey--taxis rate is large, is unknown in the literature.

Our bifurcation analysis is based on the local versions in \cite{A,CR} etc.  From the viewpoint of mathematical analysis, it is interesting to investigate the behavior or shape of these local branches, in particular in the study of positive steady states when large prey--taxis may lead to striking structures such as spikes and layers, etc.  For example, according to the global theory of Rabinowitz \cite{Ra} and its developed version in \cite{SW}, global continuum of $\Gamma_k(s)$ either intersects with the $\chi$--axis at another bifurcating point, or extends to infinity, or intersects with a \emph{singular} point.  Populations growth terms in (\ref{31}) inhibits the application of topology argument developed in \cite{CKWW,WX} etc.

When the prey--taxis rate is around bifurcation values, our results provide almost a complete understanding of the spatial--temporal dynamics of (\ref{11}) over 1D.  Further research is needed on its pattern formations when $\chi$ is away from $\chi_0$ and in particular when it is sufficiently large.  For example, rigorous analysis of the profile of the spikes obtained in numerical simulations can be an interesting problem to probe in the future.  There are also some interesting problems such as the investigation of chaotic dynamics in (\ref{11}) or bifurcation analysis of (\ref{11}) over higher dimensions.  It is also meaningful to ask about the biologically realistic traveling wave solutions to (\ref{31}), compared to those for the system without prey--taxis obtained in \cite{LWZY}.

%We comment on some notable limitations of our model and the studies.  Both our analyses of stationary and time--periodic patterns are obtained from the system over only one space dimension.  The restriction to one dimension was made because it made our calculations simpler and illustrations conceptually clearer.  Nonetheless, many of our results hold over multi--dimensional domains with $(\frac{k\pi}{L})^2$ replaced by the $k$--th Neumann eigenvalue of $-\Delta$.

%Theorem 3.2 further
%implies K2 > 0 if L is small and d1; d2 being comparably large, hence the monotone
%solution must be stable, at least when  is slightly larger than 0. On the other
%hand, when the interval length L is large or the chemical decay rate  is large, the
%small amplitude bifurcating solution (uk(s; x); vk(s; x);wk(s; x)) may be unstable.
%Therefore, it is natural to expect the formation of stable steady states of (1.1) with
%large amplitude, such as boundary layer, interior spikes, etc., or the formation of
%time--periodic spatial patterns. Rigourous analysis for these spiky solutions requires
%nontrivial mathematical tools and it is out of the scope of this paper. Numerical
%simulations are presented to illustrate the emergence of stable spiky patterns.

\section{Appendix}\label{section8}
This section is devoted to evaluating $\mathcal K_2$ in terms of system parameters in (\ref{41}).  We know from Theorem \ref{theorem52} and Theorem \ref{theorem43} that $\mathcal K_1=0$ in (\ref{418}) and the steady state bifurcation branch $\Gamma_k(s)$ is pitch--fork; moreover $\mathcal K_2$ determines the turning direction hence the stability of the steady state bifurcation $\Gamma_k(s)$ around $\chi^S_k$.    For the generality, we obtain the general expression of $\mathcal K_2$ for each branch $\Gamma_k(s)$, $k\in\mathbb N^+$.

We collect and equate $\mathcal{O}(s^3)$ terms in (\ref{41}) through (\ref{418}) to have that
\begin{align}\label{81}
\begin{split}
\begin{cases}
d_1\varphi''_2=\chi_k\bar{u}\phi(\bar{w})\gamma''_2 -\mathcal{K}_2\bar{u}\phi(\bar{w})(\frac{k\pi}{L})^2\cos{\frac{k\pi x}{L}}+({\alpha_1}\varphi_2-\beta_1\gamma_2 )\bar{u}\\
\hspace{1.4cm}+\Big((2{\alpha_1}P_k-\beta_1)\varphi_1 -\beta_1P_k\gamma_1 \Big )\cos{\frac{k\pi x}{L}}+\chi_k C_1,\\
d_2\psi''_2=\xi\bar{v}\phi(\bar{w})\gamma''_2+\Big((2{\alpha_2}Q_k-\beta_2)\psi_1 -\beta_2Q_k\gamma_1 \Big )\cos{\frac{k\pi x}{L}}\\
\hspace{1.4cm}+({\alpha_2}\psi_2-\beta_2\gamma_2 )\bar{v}+\xi C_2,\\
d_3\gamma''_2= \Big(\beta_{31}\varphi_1 +\beta_{32}\psi_1+ (\beta_{31}P_k +\beta_{32}Q_k+\frac{2\alpha_3}{K_3})\gamma_1\Big)\cos{\frac{k\pi x}{L}}\\
\hspace{1.4cm}+(\beta_{31}\varphi_2 +\beta_{32}\psi_2 {\alpha_3})\bar{w},\\
\varphi'_2(x)=\psi'_2(x)=\gamma'_2(x)=0, ~~x=0,L,
\end{cases}
\end{split}
\end{align}
where
\begin{align*}
C_1=&-\bar{u}\phi'(\bar{w})(\frac{k\pi}{L})^2\gamma_1\cos{\frac{k\pi x}{L}}-\Big(2\bar{u}\phi'(\bar{w})+P_k\phi(\bar{w})\Big)(\frac{k\pi}{L})\gamma'_1\sin{\frac{k\pi x}{L}}\\
&-\phi(\bar{w})(\frac{k\pi}{L})^2\varphi_1\cos{\frac{k\pi x}{L}}-\phi(\bar{w})(\frac{k\pi}{L})\varphi'_1\sin{\frac{k\pi x}{L}}+\Big(\bar{u}\phi'(\bar{w})\\
&+P_k\phi(\bar{w})\Big)\gamma''_1\cos{\frac{k\pi x}{L}}+\Big(\phi''(\bar{w})\bar{w}+2P_k\phi'(\bar{w})\Big)(\frac{k\pi}{L})^2\sin^2{\frac{k\pi x}{L}}\cos{\frac{k\pi x}{L}}\\
&-\Big(\frac{1}{2}\phi''(\bar{w})\bar{w}+P_k\phi'(\bar{w})\Big )(\frac{k\pi}{L})^2\cos^3{\frac{k\pi x}{L}},
\end{align*}
and
\begin{align*}
C_2=&-\bar{v}\phi'(\bar{w})(\frac{k\pi}{L})^2\gamma_1\cos{\frac{k\pi x}{L}}-\Big(2\bar{v}\phi'(\bar{w})+Q_k\phi(\bar{w})\Big)(\frac{k\pi}{L})\gamma'_1\sin{\frac{k\pi x}{L}}\\
&-\phi(\bar{w})(\frac{k\pi}{L})^2\psi_1\cos{\frac{k\pi x}{L}}-\phi(\bar{w})(\frac{k\pi}{L})\psi'_1\sin{\frac{k\pi x}{L}}+\Big(\bar{v}\phi'(\bar{w})\\
&+Q_k\phi(\bar{w})\Big)\gamma''_1\cos{\frac{k\pi x}{L}}+\Big(\phi''(\bar{w})\bar{w}+2Q_k\phi'(\bar{w})\Big)(\frac{k\pi}{L})^2\sin^2{\frac{k\pi x}{L}}\cos{\frac{k\pi x}{L}}\\
&-\Big(\frac{1}{2}\phi''(\bar{w})\bar{w}+Q_k\phi'(\bar{w})\Big )(\frac{k\pi}{L})^2\cos^3{\frac{k\pi x}{L}}.
\end{align*}
We multiply the first equation in (\ref{81}) by $\cos{\frac{k\pi x}{L}}$ and integrate it over $(0,L)$, which implies that
\begin{align*}%\label{82}
\begin{split}
\frac{\mathcal{K}_2\bar{u}\phi(\bar{w})(k\pi)^2}{2L}=&\Big(d_1(\frac{k\pi}{L})^2+{\alpha_1}\bar{u}\Big)\int_{0}^{L}\varphi_2\cos{\frac{k\pi x}{L}}dx-\Big(\chi_k\bar{u}\phi(\bar{w})(\frac{k\pi}{L})^2+\beta_1\bar{u} \Big)\\
&\cdot\int_{0}^{L}\gamma_2\cos{\frac{k\pi x}{L}}dx+\Big({\alpha_1}P_k-\frac{1}{2}\beta_1+\frac{1}{2}\chi_k\phi(\bar{w})(\frac{k\pi}{L})^2 \Big)\\
&\cdot\int_{0}^{L}\varphi_1\cos{\frac{2k\pi x}{L}}dx-\frac{1}{2}\Big(\chi_k\bar{u}\phi'(\bar{w})(\frac{k\pi}{L})^2+\beta_1P_k+\chi_kP_k\phi(\bar{w}) \Big)%\\
\end{split}
\end{align*}\begin{align}\label{82}
\begin{split}&\cdot\int_{0}^{L}\gamma_1\cos{\frac{2k\pi x}{L}}dx+\Big({\alpha_1}P_k-\frac{1}{2}\beta_1-\frac{1}{2}\chi_k\phi(\bar{w})(\frac{k\pi}{L})^2\Big)\int_{0}^{L}\varphi_1dx\\
&-\frac{1}{2}\Big( \chi_k\bar{u}\phi'(\bar{w})(\frac{k\pi}{L})^2+\beta_1P_k\Big)\int_{0}^{L}\gamma_1dx.
\end{split}
\end{align}
On the other hand, we test the second and the third equations in (\ref{81}) by same method to obtain
\begin{align}\label{83}
\begin{split}
0=&\Big(d_2(\frac{k\pi}{L})^2+{\alpha_2}\bar{v}\Big)\int_{0}^{L}\psi_2\cos{\frac{k\pi x}{L}}dx-\Big(\xi\bar{u}\phi(\bar{w})(\frac{k\pi}{L})^2+\beta_2\bar{v} \Big)\\
&\cdot\int_{0}^{L}\gamma_2\cos{\frac{k\pi x}{L}}dx+\Big({\alpha_2}Q_k-\frac{1}{2}\beta_2+\frac{1}{2}\xi\phi(\bar{w})(\frac{k\pi}{L})^2 \Big)\\
&\cdot\int_{0}^{L}\psi_1\cos{\frac{2k\pi x}{L}}dx-\frac{1}{2}\Big(\xi\bar{v}\phi'(\bar{w})(\frac{k\pi}{L})^2+\beta_2Q_k+\xi Q_k\phi(\bar{w}) \Big)\\
&\cdot\int_{0}^{L}\gamma_1\cos{\frac{2k\pi x}{L}}dx+\Big({\alpha_2}Q_k-\frac{1}{2}\beta_2-\frac{1}{2}\xi\phi(\bar{w})(\frac{k\pi}{L})^2\Big)\int_{0}^{L}\psi_1dx\\
&-\frac{1}{2}\Big( \xi\bar{u}\phi'(\bar{w})(\frac{k\pi}{L})^2+\beta_2Q_k\Big)\int_{0}^{L}\gamma_1dx,
\end{split}
\end{align}
and
\begin{align}\label{84}
\begin{split}
0=& \beta_{31}\bar{w}\int_{0}^{L}\varphi_2\cos{\frac{k\pi x}{L}}dx+\beta_{32}\bar{w}\int_{0}^{L}\psi_2\cos{\frac{k\pi x}{L}}dx +d_3(\frac{k\pi}{L})^2 \int_{0}^{L}\gamma_2\cos{\frac{k\pi x}{L}}dx \\
&+\frac{\beta_{31}}{2}\int_{0}^{L}\varphi_1\cos{\frac{2k\pi x}{L}}dx+\frac{\beta_{32}}{2}\int_{0}^{L}\psi_1\cos{\frac{2k\pi x}{L}}dx\\
&+\frac{1}{2}(\beta_{31}P_k +\beta_{32}Q_k+\frac{2\alpha_3}{K_3})\int_{0}^{L}\gamma_1\cos{\frac{2k\pi x}{L}}dx+\frac{\beta_{31}}{2}\int_{0}^{L}\varphi_1dx\\
&+\frac{\beta_{32}}{2}\int_{0}^{L}\psi_1dx+\frac{1}{2}(\beta_{31}P_k +\beta_{32}Q_k+\frac{2\alpha_3}{K_3})\int_{0}^{L}\gamma_1dx.
\end{split}
\end{align}
Moreover, from $(\varphi_2,\psi_2,\gamma_2)\in\mathcal{Z}$ defined in (\ref{414}), it follows that
\begin{align}\label{85}
P_k\int_{0}^{L}\varphi_2\cos{\frac{k\pi x}{L}}dx+Q_k\int_{0}^{L}\psi_2\cos{\frac{k\pi x}{L}}dx+\int_{0}^{L}\gamma_2\cos{\frac{k\pi x}{L}}dx=0.
\end{align}
We combine (\ref{83})--(\ref{85}) in the following system
\begin{align}\label{86}
\begin{pmatrix}
0 & d_2(\frac{k\pi}{L})^2+{\alpha_2}\bar{v} & -\xi\bar{v}\phi(\bar{w})(\frac{k\pi}{L})^2-\beta_1\bar{u}\\
\beta_{31}\bar{w} & \beta_{32}\bar{w} & d_3(\frac{k\pi}{L})^2+{\alpha_3}\bar{w}\\
P_k & Q_k & 1
\end{pmatrix}
\begin{pmatrix}
\int_{0}^{L}\varphi_2\cos{\frac{k\pi x}{L}}dx\\
\int_{0}^{L}\psi_2\cos{\frac{k\pi x}{L}}dx\\
\int_{0}^{L}\gamma_2\cos{\frac{k\pi x}{L}}dx
\end{pmatrix}
=
\begin{pmatrix}
M_1\\M_2\\0
\end{pmatrix},
\end{align}
where
\begin{align*}%\label{87}
\begin{split}
M_1=&-\Big({\alpha_2}Q_k-\frac{1}{2}\beta_2+\frac{1}{2}\xi\phi(\bar{w})(\frac{k\pi}{L})^2 \Big)\int_{0}^{L}\psi_1\cos{\frac{2k\pi x}{L}}dx\\
&+\frac{1}{2}\Big(\xi\bar{v}\phi'(\bar{w})(\frac{k\pi}{L})^2+\beta_2Q_k+\xi Q_k\phi(\bar{w}) \Big)\int_{0}^{L}\gamma_1\cos{\frac{2k\pi x}{L}}dx%\\
\end{split}
\end{align*}\begin{align}\label{87}
\begin{split}&-\Big({\alpha_2}Q_k-\frac{1}{2}\beta_2-\frac{1}{2}\xi\phi(\bar{w})(\frac{k\pi}{L})^2\Big)\int_{0}^{L}\psi_1dx\\
&+\frac{1}{2}\Big( \xi\bar{u}\phi'(\bar{w})(\frac{k\pi}{L})^2+\beta_2Q_k\Big)\int_{0}^{L}\gamma_1dx,
\end{split}
\end{align}
and
\begin{align}\label{88}
\begin{split}
M_2=&-\frac{\beta_{31}}{2}\int_{0}^{L}\varphi_1\cos{\frac{2k\pi x}{L}}dx-\frac{\beta_{32}}{2}\int_{0}^{L}\psi_1\cos{\frac{2k\pi x}{L}}dx\\
&-\frac{1}{2}(\beta_{31}P_k +\beta_{32}Q_k+\frac{2\alpha_3}{K_3})\int_{0}^{L}\gamma_1\cos{\frac{2k\pi x}{L}}dx-\frac{\beta_{31}}{2}\int_{0}^{L}\varphi_1dx\\
&-\frac{\beta_{32}}{2}\int_{0}^{L}\psi_1dx-\frac{1}{2}(\beta_{31}P_k +\beta_{32}Q_k+\frac{2\alpha_3}{K_3})\int_{0}^{L}\gamma_1dx.
\end{split}
\end{align}
Solving (\ref{86}) by Cramer's rule, we obtain that
\begin{align}\label{89}
\int_{0}^{L}\varphi_2\cos{\frac{k\pi x}{L}}dx=\frac{\vert \mathcal{A}_1 \vert}{\vert \mathcal{A}_0 \vert}, \int_{0}^{L}\psi_2\cos{\frac{k\pi x}{L}}dx=\frac{\vert \mathcal{A}_2 \vert}{\vert \mathcal{A}_0 \vert},
\end{align}
and
\begin{align}\label{810}
\int_{0}^{L}\gamma_2\cos{\frac{k\pi x}{L}}dx=\frac{\vert \mathcal{A}_3 \vert}{\vert \mathcal{A}_0 \vert}
\end{align}
where
\begin{align}\label{811}
\mathcal{A}_1=&
\begin{pmatrix}
M_1 & d_2(\frac{k\pi}{L})^2+{\alpha_2}\bar{v} & -\xi\bar{v}\phi(\bar{w})(\frac{k\pi}{L})^2-\beta_1\bar{u}\\
M_2 & \beta_{32}\bar{w} & d_3(\frac{k\pi}{L})^2+{\alpha_3}\bar{w}\\
0 & Q_k & 1
\end{pmatrix},\nonumber
\\
\mathcal{A}_2=&
\begin{pmatrix}
0 & M_1 & -\xi\bar{v}\phi(\bar{w})(\frac{k\pi}{L})^2-\beta_1\bar{u}\\
\beta_{31}\bar{w} & M_2 & d_3(\frac{k\pi}{L})^2+{\alpha_3}\bar{w}\\
P_k & 0 & 1
\end{pmatrix},
\\
\mathcal{A}_3=&
\begin{pmatrix}
0 & d_2(\frac{k\pi}{L})^2+{\alpha_2}\bar{v} & M_1\\
\beta_{31}\bar{w} & \beta_{32}\bar{w} & M_2\\
P_k & Q_k & 0
\end{pmatrix},\nonumber
\end{align}
and
\begin{align}\label{812}
\mathcal{A}_0=
\begin{pmatrix}
0 & d_2(\frac{k\pi}{L})^2+{\alpha_2}\bar{v} & -\xi\bar{v}\phi(\bar{w})(\frac{k\pi}{L})^2-\beta_1\bar{u}\\
\beta_{31}\bar{w} & \beta_{32}\bar{w} & d_3(\frac{k\pi}{L})^2+{\alpha_3}\bar{w}\\
P_k & Q_k & 1
\end{pmatrix}.
\end{align}

Due to the Neumann boundary conditions and $\mathcal{K}_1=0$, integrating (\ref{222}) by parts yields
\begin{align*}
\begin{cases}
{\alpha_1}\bar{u}\int_{0}^L \varphi_1dx-\beta_1\bar{u}\int_{0}^L \gamma_1dx=-\frac{L}{2}P_k({\alpha_1}P_k-\beta_1),\\
{\alpha_2}\bar{v}\int_{0}^L \psi_1dx-\beta_2\bar{v}\int_{0}^L \gamma_1dx=-\frac{L}{2}Q_k({\alpha_2}Q_k-\beta_2),\\
\beta_{31}\bar{w}\int_{0}^L \varphi_1dx+\beta_{32}\bar{w}\int_{0}^L \psi_1dx+{\alpha_3}\bar{w}\int_{0}^L \gamma_1dx=-\frac{L}{2}(\beta_{31}P_k+\beta_{32}Q_k+{\alpha_3}),
\end{cases}
\end{align*}
or equivalently
\begin{align}\label{813}
\begin{pmatrix}
{\alpha_1}\bar{u} & 0 & -\beta_1\bar{u}\\
0 & {\alpha_2}\bar{v} & -\beta_2\bar{v}\\
\beta_{31}\bar{w} & \beta_{32}\bar{w} & {\alpha_3}\bar{w}
\end{pmatrix}
\begin{pmatrix}
\int_{0}^L \varphi_1dx\\
\int_{0}^L \psi_1dx\\
\int_{0}^L \gamma_1dx
\end{pmatrix}
=
\begin{pmatrix}
-\frac{L}{2}P_k({\alpha_1}P_k-\beta_1)\\
-\frac{L}{2}Q_k({\alpha_2}Q_k-\beta_2)\\
-\frac{L}{2}(\beta_{31}P_k+\beta_{32}Q_k+{\alpha_3})
\end{pmatrix}.
\end{align}
Solving (\ref{813}) leads us to
\begin{align}\label{814}
\int_{0}^L \varphi_1dx=\frac{\vert \mathcal{B}_1 \vert}{\vert \mathcal{B}_0 \vert}, \int_{0}^L \psi_1dx=\frac{\vert \mathcal{B}_2 \vert}{\vert \mathcal{B}_0 \vert}, \int_{0}^L \gamma_1dx=\frac{\vert \mathcal{B}_3 \vert}{\vert \mathcal{B}_0 \vert},
\end{align}
where
\begin{align}\label{815}
\mathcal{B}_1=&
\begin{pmatrix}
-\frac{L}{2}P_k({\alpha_1}P_k-\beta_1) & 0 & -\beta_1\bar{u}\\
-\frac{L}{2}Q_k({\alpha_2}Q_k-\beta_2) & {\alpha_2}\bar{v} & -\beta_2\bar{v}\\
-\frac{L}{2}(\beta_{31}P_k+\beta_{32}Q_k+{\alpha_3}) & \beta_{32}\bar{w} & {\alpha_3}\bar{w}
\end{pmatrix},\nonumber
\\
\mathcal{B}_2=&
\begin{pmatrix}
{\alpha_1}\bar{u} & -\frac{L}{2}P_k({\alpha_1}P_k-\beta_1) & -\beta_1\bar{u}\\
0 & -\frac{L}{2}Q_k({\alpha_2}Q_k-\beta_2) & -\beta_2\bar{v}\\
\beta_{31}\bar{w} & -\frac{L}{2}(\beta_{31}P_k+\beta_{32}Q_k+{\alpha_3}) & {\alpha_3}\bar{w}
\end{pmatrix},
\\
\mathcal{B}_3=&
\begin{pmatrix}
{\alpha_1}\bar{u} & 0 & -\frac{L}{2}P_k({\alpha_1}P_k-\beta_1)\\
0 & {\alpha_2}\bar{v} & -\frac{L}{2}Q_k({\alpha_2}Q_k-\beta_2)\\
\beta_{31}\bar{w} & \beta_{32}\bar{w} & -\frac{L}{2}(\beta_{31}P_k+\beta_{32}Q_k+{\alpha_3})
\end{pmatrix},\nonumber
\end{align}
and
\begin{align}\label{816}
\mathcal{B}_0=
\begin{pmatrix}
{\alpha_1}\bar{u} & 0 & -\beta_1\bar{u}\\
0 & {\alpha_2}\bar{v} & -\beta_2\bar{v}\\
\beta_{31}\bar{w} & \beta_{32}\bar{w} & {\alpha_3}\bar{w}
\end{pmatrix}.
\end{align}
Multiplying (\ref{222}) by $\cos{\frac{2k\pi x}{L}}$ and integrating by parts yields
\begin{align}\label{817}
&\begin{pmatrix}
4d_1(\frac{k\pi}{L})^2+{\alpha_1}\bar{u} & 0 & -\chi_k\bar{u}\phi(\bar{w})(\frac{k\pi}{L})^2-\beta_1\bar{u} \\
0 & 4d_2(\frac{k\pi}{L})^2+{\alpha_2}\bar{v} & -\xi\bar{v}\phi(\bar{w})(\frac{k\pi}{L})^2-\beta_2\bar{v}\\
\beta_{31}\bar{w} & \beta_{32}\bar{w} & 4d_3(\frac{k\pi}{L})^2+{\alpha_3}\bar{w}\\
\end{pmatrix}
\begin{pmatrix}
\int_{0}^{L}\varphi_1\cos{\frac{2k\pi x}{L}}dx\\
\int_{0}^{L}\psi_1\cos{\frac{2k\pi x}{L}}dx\\
\int_{0}^{L}\gamma_1\cos{\frac{2k\pi x}{L}}dx
\end{pmatrix}\nonumber\\
&=
\begin{pmatrix}
M_3\\
M_4\\
0
\end{pmatrix},
\end{align}
where
\begin{align*}
M_3=\frac{\chi_k L}{2}(\frac{k\pi}{L})^2(\bar{u}\phi'(\bar{w})+P_k\phi(\bar{w}))-\frac{L}{4}({\alpha_1}P_k-\beta_1)P_k,
\end{align*}
and
\begin{align*}
M_4= \frac{\xi L}{2}(\frac{k\pi}{L})^2(\bar{v}\phi'(\bar{w})+Q_k\phi(\bar{w}))-\frac{L}{4}({\alpha_2}Q_k-\beta_2)Q_k.
\end{align*}

We have the solutions of (\ref{817}) that
\begin{align}\label{818}
\int_{0}^{L}\varphi_1\cos{\frac{2k\pi x}{L}}dx=\frac{\vert \mathcal{C}_1 \vert}{\vert \mathcal{C}_0 \vert},\int_{0}^{L}\psi_1\cos{\frac{2k\pi x}{L}}dx=\frac{\vert \mathcal{C}_2 \vert}{\vert \mathcal{C}_0 \vert},
\end{align}
and
\begin{align}\label{819}
\int_{0}^{L}\gamma_1\cos{\frac{2k\pi x}{L}}dx=\frac{\vert \mathcal{C}_3 \vert}{\vert \mathcal{C}_0 \vert},
\end{align}
where the notations are
\begin{align}\label{820}
\mathcal{C}_1=&
\begin{pmatrix}
M_3 & 0 & -\chi_k\bar{u}\phi(\bar{w})(\frac{k\pi}{L})^2-\beta_1\bar{u} \\
M_4 & 4d_2(\frac{k\pi}{L})^2+{\alpha_2}\bar{v} & -\xi\bar{v}\phi(\bar{w})(\frac{k\pi}{L})^2-\beta_2\bar{v}\\
0 & \beta_{32}\bar{w} & 4d_3(\frac{k\pi}{L})^2+{\alpha_3}\bar{w}\\
\end{pmatrix},\nonumber
\\
\mathcal{C}_2=&
\begin{pmatrix}
4d_1(\frac{k\pi}{L})^2+{\alpha_1}\bar{u} & M_3 & -\chi_k\bar{u}\phi(\bar{w})(\frac{k\pi}{L})^2-\beta_1\bar{u} \\
0 & M_4 & -\xi\bar{v}\phi(\bar{w})(\frac{k\pi}{L})^2-\beta_2\bar{v}\\
\beta_{31}\bar{w} & 0 & 4d_3(\frac{k\pi}{L})^2+{\alpha_3}\bar{w}\\
\end{pmatrix},
\\
\mathcal{C}_3=&
\begin{pmatrix}
4d_1(\frac{k\pi}{L})^2+{\alpha_1}\bar{u} & 0 & M_3 \\
0 & 4d_2(\frac{k\pi}{L})^2+{\alpha_2}\bar{v} & M_4\\
\beta_{31}\bar{w} & \beta_{32}\bar{w} & 0\\
\end{pmatrix},\nonumber
\end{align}
and
\begin{align}\label{821}
\mathcal{C}_0=
\begin{pmatrix}
4d_1(\frac{k\pi}{L})^2+{\alpha_1}\bar{u} & 0 & -\chi_k\bar{u}\phi(\bar{w})(\frac{k\pi}{L})^2-\beta_1\bar{u} \\
0 & 4d_2(\frac{k\pi}{L})^2+{\alpha_2}\bar{v} & -\xi\bar{v}\phi(\bar{w})(\frac{k\pi}{L})^2-\beta_2\bar{v}\\
\beta_{31}\bar{w} & \beta_{32}\bar{w} & 4d_3(\frac{k\pi}{L})^2+{\alpha_3}\bar{w}\\
\end{pmatrix}.
\end{align}
Note that $\mathcal{K}_2$ terms in (\ref{82}) consists of integrals (\ref{89})--(\ref{810}), (\ref{814}) and (\ref{818})--(\ref{819}).  Therefore, given all the system parameters, we will be able to evaluate $\mathcal{K}_2$ and determine the stability of $\Gamma_{k_0}(s)$ thanks to Theorem \ref{theorem43}.

\end{document}